\documentclass[hidelinks,onefignum,onetabnum]{siamart220329}

\usepackage{lipsum}
\usepackage{amsfonts}
\usepackage{graphicx}
\usepackage{subcaption}
\usepackage{epstopdf}
\usepackage{algorithmic}
\usepackage{cite}
\usepackage{wrapfig}
\ifpdf
  \DeclareGraphicsExtensions{.eps,.pdf,.png,.jpg}
\else
  \DeclareGraphicsExtensions{.eps}
\fi

\newsiamremark{remark}{Remark}
\newsiamremark{hypothesis}{Hypothesis}
\crefname{hypothesis}{Hypothesis}{Hypotheses}
\newsiamthm{claim}{Claim}
\newsiamthm{assum}{Assumption}
\newsiamthm{scheme}{Scheme}
\newsiamthm{lem}{Lemma}
\newsiamthm{thm}{Theorem}
\headers{Error Estimates of a Diffusion Model}{P.~Wang, Z.~Zhang, M.~Yang, F.~Bao, Y.~Cao,  G.~Zhang}

\title{Error estimates of a training-free diffusion model for high-dimensional sampling\thanks{
This material is based upon work supported by the U.S. Department of Energy, Office of Science, Office of Advanced Scientific Computing Research, Applied Mathematics program under the contract ERKJ443, Office of Fusion Energy Science, and Scientific Discovery through Advanced Computing (SciDAC) program, at the Oak Ridge National Laboratory, which is operated by UT-Battelle, LLC, for the U.S. Department of Energy under Contract DE-AC05-00OR22725. Additional support is provided by the U.S. Department of Energy under Contracts DE-SC0025412 and DE-SC0025649, and by the U.S. National Science Foundation under Grant DMS-2142672. }}

\author{Pengjun Wang\thanks{Department of Mathematics and Statistics, Auburn University, Auburn, AL.}
\and
Zezhong Zhang\footnotemark[2]
\and
Minglei Yang\thanks{Fusion Energy Division, Oak Ridge National Laboratory,
 Oak Ridge, TN.}
 \and
Feng Bao\thanks{Department of Mathematics, Florida State University,
 Tallahassee, FL.}
\and
\newline
Yanzhao Cao\footnotemark[2]
\and
Guannan Zhang\thanks{Computer Science and Mathematics Division, Oak Ridge National Laboratory, Oak Ridge, TN}
}

\usepackage{amsopn}

\usepackage{enumitem}

\ifpdf

\hypersetup{
  pdftitle={},
  pdfauthor={P.~Wang, M.~Yang, F.~Bao, Y.~Cao, G.~Zhang}
}
\fi

\usepackage{bbm}

\allowdisplaybreaks
\begin{document}

\maketitle

\begin{abstract}
Score-based diffusion models are a powerful class of generative models, but their practical use often depends on training neural networks to approximate the score function. Training-free diffusion models provide an attractive alternative by exploiting analytically tractable score functions, and have recently enabled supervised learning of efficient end-to-end generative samplers.
Despite their empirical success, the training-free diffusion models lack rigorous and numerically verifiable error estimates.
In this work, we develop a comprehensive error analysis for a class of training-free diffusion models used to generate labeled data for supervised learning of generative samplers. 
By exploiting the availability of the exact score function for Gaussian mixture models, our analysis avoids propagating score-function approximation errors through the reverse-time diffusion process and recovers classical convergence rates for ODE discretization schemes, such as first-order convergence for the Euler method. Moreover, the resulting error bounds exhibit favorable dimension dependence, scaling as $\mathcal{O}(d)$ in the $\ell_2$ norm and $\mathcal{O}(\log d)$ in the $\ell_\infty$ norm. Importantly, the proposed error estimates are fully numerically verifiable with respect to both time-step size and dimensionality, thereby bridging the gap between theoretical analysis and observed numerical behavior. 
\end{abstract}



\section{Introduction}

The diffusion models, inspired by thermodynamics modeling \cite{sohl2015deep}, have emerged as a powerful tool in the field of generative learning. 
The key idea of diffusion model is to use a forward process to progressively add noise to the training data, and learn to reverse the process to generate new samples from the target distribution. Several notable diffusion models have been developed in recent years such as denoising diffusion probabilistic models (DDPM) \cite{ho2020denoising}, denoising diffusion implicit models (DDIM) \cite{song2021denoising}, score-based generative modeling (SGM) \cite{song2020score}. In particular, the reverse process relies on access to the score of the noised data distribution, which in practice is approximated by a neural network during training. Nevertheless, learning an accurate score function with a neural network is delicate and often unstable, especially at low noise and in high dimensions. Moreover, sampling demands many reverse time steps that repeatedly call the approximate score function, making inference computationally expensive.

Beyond using training-based diffusion models, a growing line of training-free approaches seeks to run the reverse process without approximating the score function with a neural network. In effect, many of these methods are implicitly equivalent to assuming the target distribution is either a finite Gaussian mixture or a collection of point masses. While attractive for their simplicity and zero-training cost, these methods currently lack rigorous error guarantees. In our previous work in \cite{liu2024diffusion}, we proposed, for the first time, to exploit training-free score-based diffusion models to generate labeled data to learn an end-to-end generative model in a supervised learning fashion. The end-to-end generative model offers several significant advantages over diffusion models and normalizing flow models. First, unlike diffusion models, it does not require solving a reverse-time stochastic process to generate new samples, making it computationally efficient and comparable in efficiency to normalizing flow models. Second, unlike normalizing flow models, it does not require a strictly reversible architecture. This is because the availability of labeled data allows the model to be trained in a supervised learning framework, whereas normalizing flow models rely on unsupervised learning and must explicitly use the inverse mapping in the loss function. Because of these advantages, this approach has been successfully applied to amortized Bayesian inference for parameter estimation \cite{zhang2025exact, tatsuoka2025multifidelity, Fan2025GenAI4UQ, Lu2024DiffusionE3SM} and to learning transition distributions of stochastic dynamical systems \cite{Yang2026FlowMaps, Liu2025DiffusionSDE}. 
Despite the practical success of training-free diffusion models in enabling supervised learning of generative models, there remains a lack of numerically verifiable error estimates for such approaches. Although recent studies on the convergence analysis of diffusion models (see the detailed overview in Section~\ref{sec:review}) provide substantial theoretical insight, many error estimates do not reflect the error propagation behavior observed in numerical simulations. 

To bridge the gap between theory and practice, we develop rigorous error estimates for a specific class of training-free diffusion models used for supervised learning of generative samplers introduced in \cite{liu2024diffusion}. Specifically, we split the overall numerical error into three components. The first component is the data error, i.e., discrepancy between the ground truth target distribution and the Gaussian mixture model defined by the set of samples from the target distribution. The Gaussian mixture model can be viewed as an inflated empirical distribution, which is widely used  in ensemble-based uncertainty quantification, e.g., data assimilation \cite{Evensen2009,Evensen2022_Fundamentals,zhang2025iensf}. The second component comes from the discretization of the training-free diffusion model, where the Gaussian mixture model is treated as the target distribution. The third component is the error of the supervised learning of the final generative sampler using the labeled data obtained from the training-free diffusion model. Because the error bounds for the first and the third components are very well established, we focus on analyzing the second component in this work. 

Our error estimates exhibit several advantages over existing analyses in the literature. First, our approach completely avoids propagating score-function approximation errors through the reverse-time diffusion process. When a Gaussian mixture model is taken as the target distribution, the exact score function admits an analytical expression, ensuring that no score approximation error is accumulated during the diffusion. In addition, error free score function enables us to eliminate the singularity of the diffusion model.  As a result, we recover the classical convergence rates of ODE discretization schemes, such as first-order convergence when using the Euler method.
Second, the error-splitting strategy leads to a substantially improved dependence on the problem dimensionality for the second error component, namely $\mathcal{O}(d)$ in the $L_2$ norm and $\mathcal{O}(\log (d))$ in the $L_\infty$ norm. This improvement stems from avoiding the unfavorable dimension dependence typically introduced by score-function approximation errors.
Third, and most importantly, our error estimates are fully numerically verifiable, not only in terms of convergence with respect to the time-step size, but also with respect to the dimensionality.

The paper is organized as follows. In Section \ref{sec:Methodology}, we briefly introduce the score-based generative modeling and supervised learning of the generative samplers. Our main error estimates are provided in Section \ref{sec:analysis}. In Section \ref{sec:experiment}, we verify our theoretical results with comprehensive numerical experiments.

\subsection{Related work on error estimates for score-based diffusion models}\label{sec:review}

In the error analysis of score-based diffusion models, a central question is how the sampling error propagates when the reverse-time SDE (or the probability-flow ODE) is driven by an estimated score $\hat s_t(x)$ rather than the true score $s_t(x)=\nabla_x \log p_t(x)$. A common starting point is to assume an $\ell_2$-accurate score estimate along the forward process, e.g., $\mathbb{E}[\int_0^T \|\hat s_t(X_t)-s_t(X_t)\|^2\,dt] \le \varepsilon^2,$ and then derive an end-to-end discrepancy bound between the generated and target distributions. Early results under such $\ell_2$-type assumptions often exhibit constants that scale exponentially with the dimension, reflecting a pronounced curse of dimensionality~\cite{de2022convergence,block2020generative}. Subsequent work mitigates this dimension dependence by imposing stronger regularity on the score (or drift) field, most notably spatial Lipschitz continuity $\|s_t(x)-s_t(y)\|\le L\|x-y\|$. In this regime, stability estimates trade dimension dependence for dependence on the Lipschitz constant $L$, yielding bounds that are exponential in $L$~\cite{benton2023error} or, in more refined analyses, only polynomial in $L$~\cite{chen2022sampling,chen2023improved,chen2024probability}. Building on this line, some studies allow time-dependent regularity and assume controlled growth of the Lipschitz constant (or related stability parameters), for instance that it has at most linear growth in $x$ over time, to ensure integrability of stability factors along the dynamics and to extend guarantees to broader non-autonomous settings~\cite{gao2023wasserstein}. A related direction combines training error and sampling error into a single full error analysis, yielding end-to-end guarantees for diffusion-based generation~\cite{wang2024evaluating}.

To further relax explicit assumptions on the Lipschitz constant, alternative conditions have been proposed. One approach requires not only that $\hat s_t$ be close to $s_t$, but also that the Jacobian $\nabla_x \hat s_t$ be close to $\nabla_x s_t$, enabling sharper linearization-based stability arguments and avoiding worst-case dependence on a global Lipschitz constant~\cite{li2023towards,li2024accelerating,liang2025low}. Another strategy is algorithmic: early stopping (or restricting the reverse integration horizon) is used to limit the accumulation and amplification of score estimation errors, providing practical robustness without demanding strong global regularity~\cite{beyler2025convergence,benton2024nearly,huang2024convergence}. In parallel to these continuous domain strategies, algorithmic innovations have also extended to discrete-state diffusion, where high-order integrators have been developed to accelerate sampling while maintaining rigorous convergence guarantees~\cite{ren2025fast}. Finally, instead of placing direct assumptions on the score estimation error, some works exploit additional structure by assuming that the target distribution belongs to a Gaussian family~\cite{shah2023learning,liang2024unraveling,wu2024theoretical,chidambaram2024does,chen2024learning,gatmiry2024learning} (or is sufficiently close to Gaussian in an appropriate sense~\cite{li2025dimension}). Under this assumption, the score admits a special, low-complexity form (e.g., an affine function of $x$), which can be approximated more accurately from data. This structural advantage enables sharper analyses and dimension-friendly error bounds under weaker requirements on the learned score $\hat s_t$.

\section{Problem setting}
\label{sec:Methodology}
We aim  to estimate the numerical approximation error of a training-free diffusion model for supervised learning of a generative model for producing unlimited samples from $d$-dimensional random variable, defined by
\begin{equation}\label{eq:target}
    X \in \mathbb{R}^d \text{ and } X \sim p_X(x),
\end{equation}
where $p$ denotes the probability density function of the target random variable $X$. We are given a training dataset, denoted by
\begin{equation}\label{eq:data}
    \mathcal{D}_{\rm train} = \{x_1, x_2, \ldots, x_J\} \subset \mathbb{R}^d,
\end{equation}
which contains $J$ independent samples from the target distribution $p_X(x)$. From the computational perspective, the objective is to construct a parametric generative model of the following form
\begin{equation}\label{eq:transport}
    X = F(Y; \theta)\; \text{ with }\; Y \in \mathbb{R}^d,
\end{equation}
which is a transport map that transforms the reference random variable $Y$ (e.g., a standard Gaussian random variable) into the target random variable $X$. Once the parameter $\theta$ is determined through model training, unlimited samples of $X$ can be produced by first sampling from $Y$ and then passing those samples through the trained generative model in Eq.~\eqref{eq:transport}.

A key challenge in training the generative model $F(y; \theta)$ arises from the absence of labeled training data, which means there is no corresponding sample of $Y$ for each sample $x_j \in \mathcal{D}_{\rm train}$ in Eq.~\eqref{eq:data}. Consequently, the model $F(y; \theta)$ cannot be trained using supervised learning based on a mean square error (MSE) loss. Instead, an indirect loss derived from the change-of-variables formula, i.e.,
$p_X(x) = p_Y(F^{-1}(x))|\det {\rm D}(F^{-1}(x))|$,
is commonly used to train normalizing flow models in an unsupervised manner. The limitation of the unsupervised learning approach is that it requires specially designed reversible architectures for $F(y; \theta)$ to enable efficient backpropagation through the computation of $|\det {\rm D}(F^{-1}(x))|$. To address this challenge, our previous work \cite{liu2024diffusion} introduced a training-free diffusion model to generate labeled data such that  the generative model $F(y;\theta)$ can be trained in a supervised 
manner using the MSE loss, which eliminates the need for having a reversible architecture for $F(y;\theta)$ solely for training purposes.


\subsection{Overview of the training-free diffusion model}\label{sec:train_diffusion}
We briefly review the algorithm to be analyzed in this paper, namely the training-free score-based diffusion model proposed in \cite{liu2024diffusion}, which generates labeled data for the supervised learning of the generative model in Eq.~\eqref{eq:transport}. Our algorithm is based on the score-based diffusion model, which comprises a forward process and a reverse process.
The forward process maps the target random variable $X$ to a standard Gaussian variable $Y \sim \mathcal{N}(0, \mathbf{I}_d)$ through the following SDE:
\begin{equation}\label{eq:forward_sde}
dZ_t = b(t) Z_{t} dt + \sigma(t) dW_t, \quad Z_0 = X, \; Z_1 = Y, \; t \in [0,1],
\end{equation}
where $W_t$ is a standard Brownian motion. Although there are many choices for the drift and diffusion coefficients $b(t)$ and $\sigma(t)$ in Eq.~\eqref{eq:forward_sde}, we focus on analyzing the diffusion model using the following definition of the coefficients, i.e., 
\begin{equation}\label{eq:coefficients}
b(t) = \frac{d \log \alpha_t}{dt}, \quad \sigma^2(t) = \frac{d \beta_t^2}{dt} - 2 \frac{d\log \alpha_t}{dt} \beta_t^2,
\end{equation}
where the processes $\alpha_t$ and $\beta_t$ are defined by
\begin{equation}\label{eq:alpha_beta}
\alpha_t = 1-t, \quad \beta_t^2 = t, \quad t \in [0,1].
\end{equation}
The reverse SDE for generating new samples of the target distribution is defined by
\begin{equation}\label{eq:reverse_sde}
dZ_t = \left[ b(t)Z_t - \sigma^2(t) S(Z_t, t)\right] dt + \sigma(t) d\overleftarrow{W}_t, \quad Z_0 = X, \; Z_1 = Y,
\end{equation}
where $\overleftarrow{W}_t$ is a reverse Brownian motion, $b(t)$ and $\sigma(t)$ adopt the same definition as in Eq.~\eqref{eq:coefficients}, 
and $S(Z_t, t)$ is the score function defined by
\begin{equation}\label{eq:score_def}
S(z_t, t) := \nabla_z \log q_{Z_t}(z_t),
\end{equation}
with $q_{Z_t}(z_t)$ denoting the marginal probability density function of $Z_t$ at time $t$ in Eq.~\eqref{eq:forward_sde}. When the score function in Eq.~\eqref{eq:score_def} is known or can be approximated, we can generate samples of $X$ by solving the reverse SDE in Eq.~\eqref{eq:reverse_sde} starting from a sample of $Y \sim \mathcal{N}(0, \mathbf{I}_d)$ as the terminal condition. Therefore, the central task in using the diffusion model is to approximate the score function.

{\em The training-free score estimator.}
Traditional diffusion models rely on training neural networks to approximate the score function, for example via score matching \cite{JMLR:v6:hyvarinen05a}. In contrast, we leverage the analytical structure of the forward process to derive a training-free score estimator. Substituting $q_{Z_t}(z_t) = \int_{\mathbb{R}^d} q_{Z_t|Z_0}(z_t | z_0) q_{Z_0}(z_0) dz_0$ into Eq.~\eqref{eq:score_def} and using the fact that 
\begin{equation}\label{eq:conditional}
q_{Z_t|Z_0}(z_t | z_0) = \phi(z_t; \alpha_t z_0, \beta_t^2 \mathbf{I}_d),
\end{equation}
with $\phi(z_t; \alpha_t z_0, \beta_t^2 \mathbf{I}_d)$ being the probability density function of the Gaussian distribution 
$\mathcal{N}(\alpha_t z_0, \beta_t^2 \mathbf{I}_d)$, we obtain
\begin{equation}\label{eq:score_integral}
S(z_t, t) = \int_{\mathbb{R}^d} \left(-\frac{z_t - \alpha_t z_0}{\beta_t^2}\right) w_t(z_t, z_0) q_{Z_0}(z_0) dz_0,
\end{equation}
where the weight function is defined by
\begin{equation}\label{eq:weight_func}
w_t(z_t, z_0) := \frac{q_{Z_t|Z_0}(z_t | z_0)}{\int_{\mathbb{R}^d} q_{Z_t | Z_0'}(z_t | z_0') q_{Z_0}(z_0') dz_0'}.
\end{equation}
Given the dataset $\mathcal{D}_{\rm train}$ in Eq.~\eqref{eq:data} containing samples from $q_{Z_0}(z_0) := p_X(x)$, we approximate the integral in Eq.~\eqref{eq:score_integral} using the following Monte Carlo estimation:
\begin{equation}\label{eq:score_mc}
S(z_t, t) \approx \bar{S}(z_t, t) := \sum_{j=1}^{J} \left(-\frac{z_t - \alpha_t x_{j}}{\beta_t^2}\right) \bar{w}_t(z_t, x_{j}),
\end{equation}
where
\begin{equation}\label{eq:weight_mc}
\bar{w}_t(z_t, x_{j}) := \frac{q_{Z_t|Z_0}(z_t | x_{j})}{\sum_{j=1}^{J} q_{Z_t|Z_0}(z_t | x_{j})},
\end{equation}
with $q_{Z_t|Z_0}(z_t | x_{j})$ given by the Gaussian density function in Eq.~\eqref{eq:conditional}. With the score estimation in Eq.~\eqref{eq:score_mc}, we can directly solve the reverse SDE in Eq.~\eqref{eq:reverse_sde} to generate new samples of the target distribution $p_X(x)$. Nevertheless, the stochastic nature of the reverse process in Eq.~\eqref{eq:reverse_sde} cannot be used to generate labeled data for supervised learning of the map in Eq.~\eqref{eq:transport}, which will be addressed in the next subsection.

\subsection{Supervised learning of the generative model}\label{sec:learning}
We describe how to use the score representation given in Section \ref{sec:train_diffusion} to generate labeled data and use such data to train the generative model in Eq.~\eqref{eq:transport}.
Starting from the forward SDE in Eq.~\eqref{eq:forward_sde}, let $q_{Z_t}(z)$ denote the marginal probability density of $Z_t$ and write the score function as in Eq.~\eqref{eq:score_def}, i.e.,
$S(z,t):=\nabla_z \log q_{Z_t}(z)$.
The density function $q_{Z_t}$ satisfies the Fokker--Planck equation:
\begin{equation}\label{eq:fp}
  \partial_t q_{Z_t}(z) = - \nabla_z\!\cdot\!\big(b(t)\,z\,q_{Z_t}(z)\big) + \frac12 \sigma^2(t) \Delta_z q_{Z_t}(z).
\end{equation}
Using $\Delta_z q_{Z_t}=\nabla_z\!\cdot\!\big(\nabla_z q_{Z_t}\big)
=\nabla_z\!\cdot\!\big(S(z,t)\,q_{Z_t}(z)\big)$, we can rewrite Eq.~\eqref{eq:fp} in the flux form, i.e.,
\begin{equation}\label{eq:fp-score}
  \partial_t q_{Z_t}(z) = -\nabla_z\!\cdot\!\left(\big[b(t)\,z-\frac12\,\sigma^2(t)\,S(z,t)\big]\;q_{Z_t}(z)\right).
\end{equation}
Any deterministic flow $\dot z(t)=v(z(t),t)$ with the same marginals satisfies the continuity equation:
\begin{equation}\label{eq:ce}
  \partial_t q_{Z_t}(z)\;=\;-\,\nabla_z\!\cdot\!\big(v(z,t)\,q_{Z_t}(z)\big).
\end{equation}
Matching Eq.~\eqref{eq:fp-score} and Eq.~\eqref{eq:ce} yields the following form of the velocity term:
\[
  v(z,t)=b(t)\,z-\dfrac{1}{2}\,\sigma^2(t)\,S(z,t),
\]
and the probability-flow ODE driven by the velocity term is introduced as
\begin{equation}\label{eq:ode}
    dz_t = \left[ b(t)z_t - \frac{1}{2}\sigma^2(t) S(z_t, t)\right] dt,
\end{equation}
whose trajectories has the same marginal distribution as the reverse SDE in Eq.~\eqref{eq:reverse_sde}. The reverse ODE defines a much smoother functional relationship between the initial state and the terminal state than that defined by the reverse SDE in Eq.~\eqref{eq:reverse_sde}, which can be used to generate the desired labeled data.


To generate labeled data to train the generator in Eq.~\eqref{eq:transport}, we first draw $M$ samples of $Y$, denoted by $\mathcal{Y} = \{y_1, \ldots, y_M\}$ from the standard normal distribution. For $m=1, \ldots,M$, we solve the reverse ODE in Eq.~\eqref{eq:ode} and collect the state $Z_0 | y_m$, where the score function is computed using Eq.~\eqref{eq:score_mc} and Eq.~\eqref{eq:weight_mc}, and the dataset $\mathcal{D}_{\rm train}$ in Eq.~\eqref{eq:data}. The labeled training dataset is denoted by 
\begin{equation}\label{eq:label}
    \mathcal{D}_{\rm label} := \left\{ (y_m, x_m) : x_m = Z_0|y_m \, \text{ for }\; m = 1, \ldots, M\right\},
\end{equation}
where $x_m$ is obtained by solving the reverse ODE in Eq.~\eqref{eq:ode}.
%
%
After obtaining $\mathcal{D}_{\rm label}$, we can use it to train the generative model $F(y; \theta)$ in Eq.~\eqref{eq:transport} using supervised learning with the MSE loss. It has been numerically demonstrated that the supervised learning approach for training the generative models have shown promising potential in several uncertainty quantification tasks \cite{Yang2026FlowMaps, Liu2025DiffusionSDE, Fan2025GenAI4UQ, Lu2024DiffusionE3SM}. However, there is no error estimate to show the convergence rate of the discretization scheme for the training-free diffusion model, specifically for the reverse ODE in Eq.~\eqref{eq:ode}, which motivated us to conduct the error analysis presented in Section \ref{sec:analysis}.

\section{Error estimates of the training-free diffusion model}\label{sec:analysis}
The algorithm described in Section \ref{sec:learning} indicates that the accuracy of the generator $F(y;\theta)$ in Eq.~\eqref{eq:transport}, trained via supervised learning, is strongly influenced by the quality of the labeled data $\mathcal{D}_{\rm label}$ in Eq.~\eqref{eq:label}, which is obtained by solving the reverse ODE in Eq.~\eqref{eq:ode}. The main contribution of this paper is to provide error estimates for the discretization of the reverse ODE. Our error analysis seeks to answer the following central questions:
\vspace{-0.1cm}
\begin{itemize}[leftmargin=20pt]\itemsep0.2cm
    \item \emph{How does the error from the score estimation in Eq.~\eqref{eq:score_mc} accumulate through the time-stepping scheme and affect the final convergence rate?}
    
    \item \emph{Does the singularity of the drift and diffusion coefficients $b(t)$ and $\sigma(t)$ defined in Eq.~\eqref{eq:coefficients} affects the accuracy and stability in solving the reverse ODE?} 

    \item \emph{How does the discretization error of the reverse ODE depend on the dimensionality $d$ of the target random variable $X \in \mathbb{R}^d$ in Eq.~\eqref{eq:target}?}
    
\end{itemize}
\vspace{0.2cm}
To address the above questions, we first introduce an important assumption used throughout the rest of the error analysis. 

\begin{assum}\label{assum1}
Given the training dataset $\mathcal{D}_{\rm train}$ in Eq.~\eqref{eq:data}, we assume that the probability density function of the state $Z_0$ in the reverse ODE in Eq.~\eqref{eq:ode} at $t=0$ is a Gaussian mixture model defined by
\begin{equation}\label{eq:gmm}
q_{Z_0}(z_0) = \sum_{j=1}^J P_\xi(j)\, p_{Z_0|\xi}(z_0|j) =  \frac{1}{J} \sum_{j=1}^J \phi(z_0; z_0^j, \Sigma),
\end{equation}
where $\xi$ is a discrete random variable taking values in $\{1, \ldots, J\}$, $P_\xi(j) := 1/J$ is the probability of the event $\xi=j$, $z_0^j$ denotes the $j$-th sample from the training dataset $\mathcal{D}_{\rm train}$, and $p_{Z_0|\xi}(z_0|j):=\phi(z_0; z_0^j, \Sigma)$ represents the probability density function of a Gaussian distribution with mean $z_0^j$ and covariance matrix $\Sigma$. 
\end{assum}

The Gaussian mixture model in Eq.~\eqref{eq:gmm} can be interpreted as an ``inflated'' version of the empirical distribution represented by the samples in the training dataset $\mathcal{D}_{\rm train}$. Assumption \ref{assum1} is usually referred to as the \emph{inflation strategy}, which is widely used in ensemble-based data assimilation methods \cite{Evensen2009,Evensen2022_Fundamentals}. The inflation approach serves two complementary purposes in uncertainty quantification. First, it acts as a regularization mechanism that alleviates the underestimation of uncertainty commonly encountered in high-dimensional systems with limited ensemble size. By artificially broadening the ensemble spread, inflation mitigates variance collapse and promotes better
representation of the true uncertainty. Second, from a generative AI perspective, inflation helps alleviate model memorization by encouraging diversity in the learned representations. By introducing controlled stochasticity or variance inflation into the training dataset, we can reduce overfitting to training samples and enhances generalization, enabling more reliable synthesis of unseen states or scenarios.

Based on Assumption \ref{assum1}, we can split the total discretization error of the reverse ODE in Eq.~\eqref{eq:ode} into the following three parts:
\begin{equation}\label{eq:error_split}
\begin{aligned}
    & \;z_0^{\rm true}(z_1) - F(z_1;\theta) \\[5pt]
    = &\; \underbrace{z_0^{\rm true}(z_1) - z_0(z_1)}_{\text{Data error}}\;\; +\;\; \underbrace{z_0(z_1) - z^{K}(z_1)}_{\shortstack{\scriptsize Reverse ODE \\ \scriptsize  discretization error}}\;\; +\; \underbrace{z^{K}(z_1) - F(z_1;\theta)}_{\text{\scriptsize Supervised learning error}}, 
\end{aligned}
\end{equation}
where $z_0^{\rm true}(z_1)$ is ground truth solution of the reverse ODE in Eq.~\eqref{eq:ode} with $S(z_t,t)$ being the score function in Eq.~\eqref{eq:score_integral} defined by the target density function 
$p_X(x)$ in Eq.~\eqref{eq:target}, $F(z_1;\theta)$ is the predicted solution of the generative model in Eq.~\eqref{eq:transport} trained by the supervised learning approach discussed in Section \ref{sec:learning}; $z_0(z_1)$ is the solution of the reverse ODE in Eq.~\eqref{eq:ode} with $S(z_t,t)$ being the score function in Eq.~\eqref{eq:score_integral} defined by the Gaussian mixture model in Eq.~\eqref{eq:gmm}, $z^K(z_1)$ is the approximation of $z_0(z_1)$ using the numerical discretization scheme introduced in Section \ref{sec:discrete} with $K$ discrete time steps. 

We observe that data error in Eq.\eqref{eq:error_split} is solely determined by Assumption \ref{assum1} and the size of the training set $\mathcal{D}_{\rm train}$ and the supervised learning error is solely determined by classic function approximation theory. Importantly, neither of these error components depend on the diffusion model itself. Consequently, in order to address the three questions posed at the beginning of Section~\ref{sec:analysis} concerning the training-free diffusion model, we focus the remainder of Section~\ref{sec:analysis} on analyzing the reverse ODE discretization error. On the other hand, the error decomposition in Eq.~\eqref{eq:error_split} plays a crucial role in achieving only mild dependence of the resulting error bounds on the dimensionality, and in rendering the bounds fully numerically verifiable.

\subsection{Discretization of the reverse ODE}\label{sec:discrete}
%


We describe the discretization of the reverse ODE in Eq.~\eqref{eq:ode} and show that the  singularity in the drift and diffusion coefficients can be removed once the score function is incorporated, so it poses no numerical difficulty. Substituting the expressions for $b(t)$ and $\sigma(t)$ from Eq.~\eqref{eq:coefficients} into Eq.~\eqref{eq:ode}, we obtain
\begin{equation}\label{eq:ode3.1}
    dz_t =\Bigg[\frac{1}{t-1} z_t-\frac{1+t}{2(1-t)}{S}(z_t,t)\Bigg]dt, 
\end{equation}
where both the drift and diffusion coefficients are singular as $t\rightarrow 1$. 
%
%
Since the solutions to the reverse ODE in Eq.~\eqref{eq:ode3.1} and the reverse SDE in Eq.~\eqref{eq:forward_sde} share the same marginal distributions, we have
\begin{equation}
    Z_t = (1-t)Z_0 + \sqrt{t}\varepsilon_0,
\end{equation}
where the random variable $\varepsilon_0 \sim \mathcal{N}(0, \mathbf{I}_d)$ is independent with $Z_0$ and the latent variable $\xi$ in Eq.~\eqref{eq:gmm}. The variable $Z_t$, conditionally on $\xi$, is expressed as
\begin{equation}
    Z_t|\xi = (1-t)Z_0|\xi + \sqrt{t}\varepsilon_0.
\end{equation}
Then we can write the random vector $(Z_0|\xi, Z_t|\xi)^T$ jointly as
\begin{equation}\label{eq:z0zt|xi}
    \begin{aligned}
        \begin{pmatrix}
        Z_0|\xi\\
        Z_t|\xi
        \end{pmatrix}
        &=
        \begin{pmatrix}
            \mathbf{I}_d & 0\\
            (1-t)\mathbf{I}_d & \sqrt{t}\mathbf{I}_d
        \end{pmatrix}
        \begin{pmatrix}
            Z_0|\xi\\
            \varepsilon_0
        \end{pmatrix}.\\
    \end{aligned}
\end{equation}
Combining Eq.~\eqref{eq:gmm} and Eq.~\eqref{eq:z0zt|xi}, we have
\begin{equation}\label{barz0zt_2}
        \begin{pmatrix}
        Z_0|\xi = j\\
        Z_t|\xi = j
        \end{pmatrix}
        \sim
        \begin{pmatrix}
            \mathbf{I}_d & 0\\
            (1-t)\mathbf{I}_d & \sqrt{t}\mathbf{I}_d
        \end{pmatrix}
        \mathcal{N}\left(
        \begin{pmatrix}
            z_0^j\\
            0
        \end{pmatrix},
        \begin{pmatrix}
            \Sigma & 0\\
            0 & \mathbf{I}_d
        \end{pmatrix}\right),
\end{equation}
from which we have
$ Z_t|(\xi = j)
        \sim
        \mathcal{N}((1-t)z_0^j,(1-t)^2\Sigma +t\mathbf{I}_d)$. 
The marginal probability density of $Z_t$ is the equally weighted mixture defined by
\begin{equation}\label{eq:Pzt}
\begin{aligned}
     p_{Z_t}(z_t) = &\; \sum_{j=1}^{J} P_{Z_t | \xi}(z_0 \mid j) P_\xi(j)\\
     = &\;  \frac{1}{N} \sum_{j=1}^{J} \phi\left(z_t ; (1-t)z_0^j,(1-t)^2\Sigma +t\mathbf{I}_d\right),
\end{aligned}
\end{equation}
where $\phi(z_t ; (1-t)z_0^j,(1-t)^2\Sigma +t\mathbf{I}_d)$ denotes the probability density function of a Gaussian distribution with mean $(1-t)z_0^j$ and covariance matrix $(1-t)^2\Sigma +t\mathbf{I}_d$.

According to Eq.~\eqref{eq:Pzt}, we can derive the exact score function $S$ in Eq.~\eqref{eq:score_def} as
\begin{equation}\label{eq:score_1}
\begin{aligned}
    S(z_t,t)
    =&\; -\left((1-t)^2\Sigma +t\mathbf{I}_d\right)^{-1}z_t\\ 
    & \;+ (1-t)\left((1-t)^2\Sigma +t\mathbf{I}_d\right)^{-1}\sum_{j=1}^{J}z_0^j\,w_j(z_t,t).
\end{aligned}
\end{equation}
For each component index $j=1, \dots, J$, we define a simplified notation
\begin{equation}\label{eq:ej}
    e_j(z_t,t)
    := \exp\!\Big(
        -\tfrac{1}{2}\big(z_t-(1-t)z_0^j\big)^\top
        \big((1-t)^2\Sigma +t\mathbf{I}_d\big)^{-1}
        \big(z_t-(1-t)z_0^j\big)
    \Big),
\end{equation}
such that the corresponding normalized weights in Eq.~\eqref{eq:score_1} can be expressed by
\begin{equation}\label{eq:wi}
    w_j(z_t,t)
    = \frac{e_j(z_t,t)}{\sum_{j'=1}^{N} e_{j'}(z_t,t)}.
\end{equation}

On the other hand, because $\Sigma$ is positive definite, $((1-t)^2\Sigma + t\mathbf{I}_d)$ is positive definite for all $t\in[0,1]$, and hence invertible. In particular, the score function has no singularity at $t=0$.
The associated probability flow ODE in Eq.~\eqref{eq:ode3.1} becomes
\begin{equation}\label{eq:ode3.1.1}
\begin{aligned}
    dz_t =&\Bigg[\frac{1}{t-1} z_t-\frac{1+t}{2(1-t)}{S}(z_t,t)\Bigg]dt\\[8pt]
    =&\Bigg[\frac{1}{t-1} z_t-\frac{1+t}{2(1-t)}\Bigg(-((1-t)^2\Sigma +t\mathbf{I}_d)^{-1}z_t\\
    &+ (1-t)((1-t)^2\Sigma +t\mathbf{I}_d)^{-1}\sum_{j=1}^{J}z_0^jw_j(z_t,t)\Bigg)\Bigg]dt\\[8pt]
    =&\Bigg[\frac{1}{2}(\mathbf{I}_d-2(1-t)\Sigma)((1-t)^2\Sigma +t\mathbf{I}_d)^{-1}z_t\\
    &- \frac{1+t}{2}((1-t)^2\Sigma +t\mathbf{I}_d)^{-1}\sum_{j=1}^{J}z_0^jw_j(z_t,t)\Bigg]dt\\
\end{aligned}   
\end{equation}
with the initial condition $z_1\sim \mathcal{N}(0, \mathbf{I}_d)$.
%
%
Although the right-hand side of Eq.~\eqref{eq:ode3.1.1} appears singular as $t \to 1$, substituting the score expression in Eq.~\eqref{eq:score_1} shows that the $(t-1)^{-1}$ terms are fully compensated by corresponding terms in the score. Consequently, the drift coefficients extend continuously to $t=1$, and the apparent singularity is entirely removable at the coefficient level.

To compute numerical solutions of the ODE in Eq.~\eqref{eq:ode3.1.1}, we adopt the explicit Euler method. The resulting discretization scheme takes the form
\begin{equation}\label{eq:euler}
\begin{aligned}
    z^{k+1} =& z^k - \Bigg[\frac{1}{2}(\mathbf{I}_d-2kh\Sigma)(k^2h^2\Sigma +(1-kh)\mathbf{I}_d)^{-1}z^k\\
    &- \frac{2-kh}{2}(k^2h^2\Sigma +(1-kh)\mathbf{I}_d)^{-1}\sum_{j=1}^{J}z_0^jw_j(z^k,(1-kh))\Bigg]h,\\
\end{aligned}   
\end{equation}
where $K = 1/h,\ k = 0, 1, \ldots, K-1$ and the initial state $z^0$ is sampled from the standard normal distribution $\mathcal{N}(0, \mathbf{I}_d)$. 
\begin{remark}
    We emphasize that the error splitting approach in Eq.~\eqref{eq:error_split} and the Gaussian mixture assumption in Eq.~\eqref{eq:gmm} ensures that the exact score function of the diffusion model can be analytically derived, i.e., Eq.~\eqref{eq:score_1}, such that the exact score can be used in solving the reverse ODE to avoid the accumulation of the score estimation error. In fact, by comparing the Monte Carlo estimator in Eq.~\eqref{eq:score_mc} and the exact score in Eq.~\eqref{eq:score_1}, we can quickly derive that Eq.~\eqref{eq:score_1} converges to Eq.~\eqref{eq:score_mc} as $\Sigma \rightarrow 0$. In other words, Eq.~\eqref{eq:score_mc} can also be viewed as the exact score when the target distribution is the empirical distribution defined by $\mathcal{D}_{\rm train}$ without any smoothing. This answers the first question presented at the beginning of Section \ref{sec:analysis}. On the other hand, Eq.~\eqref{eq:ode3.1.1} answers the first question regarding the singularity of the diffusion model presented in Section \ref{sec:train_diffusion}. We will answer the third questions in the next subsection.
\end{remark}

 \subsection{Error estimates for the discretized reverse ODE}\label{sec:error_bounds}

We derive global error bounds for the reverse ODE discretization error defined in Eq.~\eqref{eq:error_split}. Our analysis builds on the mixture surrogate introduced in Assumption~\ref{assum1}, which smooths the empirical distribution and ensures a well-behaved score function throughout the reverse diffusion process. To facilitate the analysis, we impose the following assumption on the boundedness of the training dataset $\mathcal{D}_{\rm train}$.
\begin{assum}\label{assum2}
The samples in the training dataset $\mathcal{D}_{\rm train}$ in Eq.~\eqref{eq:data} satisfy
\[
   \big\|z_0^j\big\|_2 \le M < \infty, \quad j=1,\dots,J,
\]
where $M$ is the radius of the smallest $\ell_2$ ball that can cover the entire training set. 
\end{assum}
%
%
%
Under this assumption, the score function and drift field in Eq.~\eqref{eq:ode} admit uniform 
Lipschitz bounds in space and controlled variation in time, enabling 
dimension-explicit global error estimates for the explicit Euler scheme. Our error bounds are derived in two steps. The first step, presented in Section \ref{sec:path}, is to analyze the path-wise error for a single trajectory of the reverse ODE; the second step, presented in Section \ref{sec:ode_error}, is to analyze the error under expectation.

\subsubsection{The error bound for a single trajectory of the reverse ODE}\label{sec:path}
Before deriving the global error of the discretization scheme in Eq.~\eqref{eq:euler} for a given terminal value $z_1$, we first bound the entire exact path $t\mapsto z_t$ in terms of $z_1$. This pathwise envelope clarifies how the solution remains confined as it propagates backward from $t=1$ to $t=0$, and it prevents growth of constants that would otherwise obscure the accumulation of local discretization errors. With this bound in place, all subsequent estimates for the Euler iterates can be stated with constants that depend only on the ODE data and on $\|z_1\|_2$, ensuring a clean and stable comparison between the numerical trajectory and the exact one in the entire time domain $[0,1]$.
\begin{lem}\label{lem_zt_bound}
Under Assumption~\ref{assum1} and  \ref{assum2}, the solution $z_t$ of the reverse ODE in Eq.~\eqref{eq:ode3.1.1} satisfies 
\begin{equation}
    \|z_t\|_2 \le \frac{\max\left\{\sqrt{\lambda_{\max}}, 1\right\}}{\min\left\{\sqrt{\lambda_{\min}}, \frac14\right\}}\bigl(M+\|z_1\|_2\bigr),  \quad \forall \ 0\leq t\leq 1,
\end{equation}
where $M$ is the bound in Assumption~\ref{assum2}, $\lambda_{\min}$ and $\lambda_{\max}$ are the smallest and the largest eigenvalues of the covariance matrix $\Sigma$ of the Gaussian mixture model defined in Eq.~\eqref{eq:gmm}.
\end{lem}

\begin{proof}
For notational simplicity, we rewrite the reverse ODE in Eq.~\eqref{eq:ode3.1.1} as the following form
\begin{equation}\label{ode3.3}
    dz_t = [A(t)z_t + g(z_t,t)]\,dt,
\end{equation}
where the function $A(t)$ is defined by
\begin{equation}\label{ode3.3_A}
    A(t) = \frac{1}{2}(\mathbf{I}_d-2(1-t)\Sigma)((1-t)^2\Sigma +t\mathbf{I}_d)^{-1},
\end{equation}
and the function $g(z_t,t)$ is defined by
\begin{equation}\label{ode3.3_g}
    g(z_t, t) = - \dfrac{1+t}{2}((1-t)^2\Sigma +t\mathbf{I}_d)^{-1}\sum_{j=1}^{J}z_0^jw_j(z_t,t).
\end{equation}
By the variation of parameters formula, the solution of reverse ODE in Eq.~\eqref{ode3.3} can be written as
\begin{equation}\label{ode3.3_zt}
\begin{aligned}
    z_t
    =&\exp\left(\int_1^t A(s)ds\right)z_1
    - \int_t^1\exp\left(\int_s^t A(u)du\right)g(z_s,s)\,ds .
\end{aligned}
\end{equation}
Using the explicit form of the fundamental matrix associated with $A(t)$ in Eq.~\eqref{ode3.3_A}, we have
\begin{equation}\label{ode3.3_zt_2}
\begin{aligned}
    z_t
    =&\bigl((1-t)^2\Sigma +t\mathbf{I}_d\bigr)^{\frac{1}{2}}z_1 \\
    & -\int_t^1\bigl((1-t)^2\Sigma +t\mathbf{I}_d\bigr)^{\frac{1}{2}}
    \bigl((1-s)^2\Sigma +s\mathbf{I}_d\bigr)^{-\frac{1}{2}}
    g(z_s,s)\,ds .
\end{aligned}
\end{equation}
Substituting the definition of $g(z_s,s)$ in Eq.~\eqref{ode3.3_g} into Eq.~\eqref{ode3.3_zt_2}, we have
\begin{equation}\label{ode3.3_zt_3}
\begin{aligned}
    z_t
    =&\bigl((1-t)^2\Sigma +t\mathbf{I}_d\bigr)^{\frac{1}{2}}z_1 \\
    &+\bigl((1-t)^2\Sigma +t\mathbf{I}_d\bigr)^{\frac{1}{2}}
    \int_t^1\frac{1+s}{2}\bigl((1-s)^2\Sigma +s\mathbf{I}_d\bigr)^{-\frac{3}{2}}
    \sum_{j=1}^{J}z_0^j\,w_j(z_s,s)\,ds .
\end{aligned}
\end{equation}
Taking the $\ell_2$ norm on both sides of Eq.~\eqref{ode3.3_zt_3} and using the triangle inequality and sub-multiplicativity of matrix norms, we obtain
\begin{equation}\label{ode3.3_|zt|_1}
\begin{aligned}
    \|z_t\|_2\le&
    \left\|\bigl((1-t)^2\Sigma +t\mathbf{I}_d\bigr)^{\frac{1}{2}}\right\|_2\,\|z_1\|_2 + \left\|\bigl((1-t)^2\Sigma +t\mathbf{I}_d\bigr)^{\frac{1}{2}}\right\|_2\\
    &\cdot\int_t^1\frac{1+s}{2}\left\|\bigl((1-s)^2\Sigma +s\mathbf{I}_d\bigr)^{-\frac{3}{2}}\right\|_2
    \sum_{j=1}^{J}\left\|z_0^j\right\|_2\,w_j(z_s,s)\,ds .
\end{aligned}
\end{equation}
Next we bound each term on the left of Eq.~\eqref{ode3.3_|zt|_1} uniformly in $d$. Since $\Sigma$ is positive definite with eigenvalues in $[\lambda_{\min},\lambda_{\max}]$ and $0 \le t \le 1$, it holds that
\begin{equation*}
    \left\|\bigl((1-t)^2\Sigma +t\mathbf{I}_d\bigr)^{\frac{1}{2}}\right\|_2
\le \bigl((1-t)^2\lambda_{\max}+t\bigr)^{\frac{1}{2}},
\end{equation*}
and
\begin{equation*}
    \left\|\bigl((1-s)^2\Sigma +s\mathbf{I}_d\bigr)^{-\frac{3}{2}}\right\|_2
\le \bigl((1-s)^2\lambda_{\min}+s\bigr)^{-\frac{3}{2}}.
\end{equation*}
Moreover, by Assumption~\ref{assum2}, we have
\begin{equation*}
    \sum_{j=1}^{J}\left\|z_0^j\right\|_2\,w_j(z_s,s)\le M\sum_{j=1}^J w_j(z_s,s)=M.
\end{equation*}
Substituting the above bounds into Eq.~\eqref{ode3.3_|zt|_1} gives
\begin{equation}\label{ode3.3_|zt|_2}
\begin{aligned}
    \|z_t\|_2 \le&
    \bigl((1-t)^2\lambda_{\max}+t\bigr)^{\frac{1}{2}}\|z_1\|_2 \\
    &+\bigl((1-t)^2\lambda_{\max}+t\bigr)^{\frac{1}{2}}
    \int_t^1\frac{1+s}{2}\bigl((1-s)^2\lambda_{\min}+s\bigr)^{-\frac{3}{2}}\, M\,ds .
\end{aligned}
\end{equation}
Finally, the integral term in Eq.~\eqref{ode3.3_|zt|_2} can be evaluated explicitly, leading to
\begin{equation}\label{ode3.3_|zt|_3}
\begin{aligned}
    \|z_t\|_2\le&
    \bigl((1-t)^2\lambda_{\max}+t\bigr)^{\frac{1}{2}}\|z_1\|_2\\
    &+\bigl((1-t)^2\lambda_{\max}+t\bigr)^{\frac{1}{2}}
    \bigl((1-t)^2\lambda_{\min}+t\bigr)^{-\frac{1}{2}}(1-t)\,M .
\end{aligned}
\end{equation}
\vspace{0.3cm}
Since $\Sigma$ is positive definite and $0\le t\le 1$, we further have 
\begin{equation}\label{lem_zt_bound_aux1}
\begin{aligned}
\max_{0\le t\le 1}\Bigl((1-t)^2\lambda_{\max}+t\Bigr)
=&
\begin{cases}
\lambda_{\max}, & \lambda_{\max}\ge 1,\\[10pt]
1, & 0<\lambda_{\max}< 1,
\end{cases}\\[10pt]
\le&\max\{\lambda_{\max},\,1\}.
\end{aligned}
\end{equation}
and we have
\begin{equation}\label{lem_zt_bound_aux2}
\begin{aligned}
\max_{0\le t\le 1}\Bigl((1-t)^2\lambda_{\min}+t\Bigr)^{-1}
=&
\begin{cases}
\dfrac{1}{\lambda_{\min}}, & 0<\lambda_{\min}\le \dfrac12,\\[10pt]
\dfrac{4\lambda_{\min}}{4\lambda_{\min}-1}, & \lambda_{\min}>\dfrac12,
\end{cases}\\[10pt]
\le& \max\left\{\dfrac{1}{\lambda_{\min}},\,2\right\}
= \dfrac{1}{\min\{\lambda_{\min},\,\tfrac12\}}.
\end{aligned}
\end{equation}
Using the bounds in Eq.~\eqref{lem_zt_bound_aux1} and Eq.~\eqref{lem_zt_bound_aux2}, we immediately obtain
\begin{equation}\label{lem_zt_bound_conseq}
\begin{aligned}
&\bigl((1-t)^2\lambda_{\max}+t\bigr)^{\frac12}
\le \max\left\{\sqrt{\lambda_{\max}},\,1\right\},\\
&\bigl((1-t)^2\lambda_{\min}+t\bigr)^{-\frac12}
\le \frac{1}{\min\{\sqrt{\lambda_{\min}},\,\tfrac14\}}.
\end{aligned}
\end{equation}
Substituting these bounds into Eq.~\eqref{ode3.3_|zt|_3} and using $1-t\le 1$, we obtain
\begin{equation}\label{lem_zt_bound_final}
\begin{aligned}
\|z_t\|_2
\le&\max\{\sqrt{\lambda_{\max}},\,1\}\,\|z_1\|_2
+\frac{\max\{\sqrt{\lambda_{\max}}, 1\}}{\min\{\sqrt{\lambda_{\min}}, \frac14\}}\,M\\
\le&\frac{\max\{\sqrt{\lambda_{\max}}, 1\}}{\min\{\sqrt{\lambda_{\min}}, \frac14\}}\bigl(\|z_1\|_2+M\bigr),
\end{aligned}
\end{equation}
which completes the proof.
\end{proof}

Next, we derive a uniform-in-time Lipschitz constant in $z_t$ for the right-hand side of the reverse ODE in Eq. \eqref{eq:ode3.1.1}, and the bound of its second derivatives with respect to $t$.
For this purpose we present two lemmas that provide derivative estimates for the weights $w_j(z_t,t)$ in Eq.~\eqref{eq:wi}.





\begin{lem}\label{lem_dwidzt}
Under Assumption \ref{assum1} and \ref{assum2}, for each $j = 1, \ldots, J$, the gradient of $w_j(z_t, t)$ in Eq.~\eqref{eq:wi} with respect to $z_t$ satisfies
\begin{equation}
    \begin{aligned}
        &\left\|\nabla_{z_t}w_j(z_t,t)\right\|_2\le\frac{2M}{\min\{\lambda_{\min},\frac{1}{2}\}}w_j(z_t,t),
    \end{aligned}
\end{equation}
where $M$ is the bound in Assumption~\ref{assum2} and $\lambda_{\min}$ denotes the minimal eigenvalue of the covariance matrix $\Sigma$ of the Gaussian mixture model in Eq.~\eqref{eq:gmm}.

\end{lem}

\begin{proof}
From the definition of $w_j(z_t,t)$ given in Eq.~\eqref{eq:wi}, we have
\begin{equation*}
    \begin{aligned}
        &\nabla_{z_t}w_j(z_t,t)\\
        =&\;\;\nabla_{z_t}\frac{\exp (-\frac{1}{2}(z_t-(1-t)z_0^j)^T((1-t)^2\Sigma +t\mathbf{I}_d)^{-1}(z_t-(1-t)z_0^j))}{\sum_{j'=1}^{J}\exp (-\frac{1}{2}(z_t-(1-t)z_0^{j'})^T((1-t)^2\Sigma +t\mathbf{I}_d)^{-1}(z_t-(1-t)z_0^{j'}))}\\
        =&\;\;\nabla_{z_t}\frac{e_j(z_t,t)}{\sum_{j'=1}^{J}e_{j'}(z_t,t)}\\
        =&\;\;\frac{((1-t)^2\Sigma +t\mathbf{I}_d)^{-1}((1-t)z_0^j-z_t)e_j(z_t,t)}{\sum_{j'=1}^{J}e_{j'}(z_t,t)}\\
        &-\;\;\frac{e_j(z_t,t)\left(\sum_{j'=1}^{J}((1-t)^2\Sigma +t\mathbf{I}_d)^{-1}((1-t)z_0^{j'}-z_t)e_{j'}(z_t,t)\right)}{\left(\sum_{j'=1}^{J}e_{j'}(z_t,t)\right)^2}\\
        =&\;\;((1-t)^2\Sigma +t\mathbf{I}_d)^{-1}w_j(z_t,t)\left[((1-t)z_0^j-z_t) -  \sum_{j'=1}^{J}((1-t)z_0^{j'}-z_t)w_{j'}(z_t,t)\right].\\
    \end{aligned}
\end{equation*}
Noting that $\sum_{j'=1}^{J} w_{j'}(z_t,t) = 1$, we can simplify the above equation as
\begin{equation*}
    \begin{aligned}
        \nabla_{z_t}w_j(z_t,t)
        =&\,((1-t)^2\Sigma +t\mathbf{I}_d)^{-1}w_j(z_t,t)\\
        &\cdot\left[\sum_{j'=1}^{J}((1-t)z_0^j-z_t)w_{j'}(z_t,t) - \sum_{j'=1}^{J}((1-t)z_0^{j'}-z_t)w_{j'}(z_t,t)\right]\\
        =&\, (1-t)((1-t)^2\Sigma +t\mathbf{I}_d)^{-1}w_j(z_t,t)\sum_{j'=1}^{J}(z_0^j-z_0^{j'})w_{j'}(z_t,t).
    \end{aligned}
\end{equation*}
Taking the norm on both sides yields the estimate
\begin{equation}\label{grad_wi}
    \begin{aligned}
        &\left\|\nabla_{z_t}w_j(z_t,t)\right\|_2
        \le\left\|(1-t)((1-t)^2\Sigma +t\mathbf{I}_d)^{-1} \right\|_2w_j(z_t,t)\sum_{j'=1}^{J}\left\|z_0^j-z_0^{j'}\right\|_2w_{j'}(z_t,t).\\
    \end{aligned}
\end{equation}
Since $\Sigma$ is positive definite and $0\le t\le 1$, we have 
\begin{equation}\label{lem_dwidzt_eq1}
\begin{aligned}
&\max_{0\le t\le 1}\Bigl\|\bigl(1-t)((1-t)^2\Sigma+t\mathbf{I}_d\bigr)^{-1}\Bigr\|_2
=\max_{0\le t\le 1}\frac{1-t}{(1-t)^2\lambda_{\min}+t}
\le\frac{1}{\min\{\lambda_{\min},\,\tfrac12\}}, 
\end{aligned}
\end{equation}
where the last inequality follows from Eq.~\eqref{lem_zt_bound_aux2}.
Furthermore, we have $\|z_0^j\|_2< M$ for all $j$ in Assumption \eqref{assum2}. Hence, for all $j,j'$, $\|z_0^j - z_0^{j'}\|_2 \le 2M$. Substituting this into Eq.~\eqref{grad_wi}, and applying Eq.~\eqref{lem_dwidzt_eq1}, we deduce
\begin{equation}\label{lem_dwidzt_eq3}
    \begin{aligned}
        \left\|\nabla_{z_t}w_j(z_t,t)\right\|_2 \le& \frac{2M}{\min\{\lambda_{\min}, \frac{1}{2}\}}w_j(z_t,t)\sum_{j=1}^{J}w_j(z_t,t)
        =\frac{2M}{\min\{\lambda_{\min},\frac{1}{2}\}}w_j(z_t,t).
    \end{aligned}
\end{equation}
This completes the proof.
\end{proof}

\begin{lem}\label{lem_wi_t}
Under Assumptions~\ref{assum1} and \ref{assum2}, for each $j = 1, \ldots, J$, the time derivative of $w_j(z_t,t)$ in Eq.~\eqref{eq:wi} satisfies
\begin{equation*}
    \begin{aligned}
        &\left|\frac{\partial w_j(z_t,t)}{\partial t}\right|\le\dfrac{2+2\lambda_{\max}}{\min\{\lambda^2_{\min},\frac{1}{4}\}}(\|z_t\|_2 + M)^2w_j(z_t,t),
    \end{aligned}
\end{equation*}
where $M$ is the data bound in Assumption~\ref{assum2}, $\lambda_{\min}$ and $\lambda_{\max}$ are resepectively the smallest and largest eigenvalues of the covariance matrix $\Sigma$ from the Gaussian mixture model in Eq.~\eqref{eq:gmm}.
\end{lem}
\begin{proof}

We begin by estimating the time derivative of $e_j(z_t, t)$ in Eq.~\eqref{eq:ej}:
\begin{equation*}
    \begin{aligned}
        \left|\frac{\partial e_j(z_t, t)}{\partial t}\right|
        =& \left| \frac{\partial}{\partial t} \exp\left(-\frac{1}{2}(z_t-(1-t)z_0^j)^T ((1-t)^2\Sigma + t\mathbf{I}_d)^{-1}(z_t-(1-t)z_0^j) \right) \right| \\
        =&\; e_j(z_t, t) \Big| -(z_t - (1-t)z_0^j)^T ((1-t)^2\Sigma + t\mathbf{I}_d)^{-1} z_0^j \\
        &+ \frac{1}{2}(z_t - (1-t)z_0^j)^T ((1-t)^2\Sigma + t\mathbf{I}_d)^{-1} (\mathbf{I}_d - 2(1-t)\Sigma) \\
        &\quad \cdot ((1-t)^2\Sigma + t\mathbf{I}_d)^{-1} (z_t - (1-t)z_0^j) \Big|.
    \end{aligned}
\end{equation*}
By bounding each term using $\|((1-t)^2\Sigma + t\mathbf{I}_d)^{-1}\|_2$ $ \le \dfrac{1}{\min\{\lambda_{\min},\frac{1}{2}\}}$ in Eq.~\eqref{lem_zt_bound_aux2} and $\|z_0^j\|_2 \le M$ in Assumption \ref{assum2}, we obtain
\begin{equation*}
    \begin{aligned}
        \left|\frac{\partial e_j(z_t, t)}{\partial t}\right|
        \le& e_j(z_t, t) \left( \frac{(\|z_t\|_2 + M) M}{\min\{\lambda_{\min},\frac{1}{2}\}} + \frac{(\|z_t\|_2 + M)^2(1+2\lambda_{\max})}{2\min\{\lambda^2_{\min},\frac{1}{4}\}} \right)\\
        \le&e_j(z_t, t)\left(\frac{(\|z_t\|_2 + M)^2}{2\min\{\lambda^2_{\min},\frac{1}{4}\}} + \frac{(\|z_t\|_2 + M)^2(1+2\lambda_{\max})}{2\min\{\lambda^2_{\min},\frac{1}{4}\}}\right)\\
        =&e_j(z_t, t)\frac{1+\lambda_{\max}}{\min\{\lambda^2_{\min},\frac{1}{4}\}}(\|z_t\|_2 + M)^2.
    \end{aligned}
\end{equation*}
For  the derivative of $w_j(z_t, t)$ in Eq.~\eqref{eq:wi}, we have 
\begin{equation}\label{dwidt}
    \begin{aligned}
        \left|\frac{\partial w_j(z_t,t)}{\partial t}\right|
        =&\left|\frac{\partial}{\partial t}\frac{e_j(z_t, t)}{\sum_{j'=1}^J e_{j'}(z_t, t)}\right|\\
        \le&\left|\dfrac{\dfrac{\partial e_j(z_t, t)}{\partial t}}{\sum_{j'=1}^J e_{j'}(z_t, t)}\right| + \left|\frac{e_j(z_t, t)\sum_{j'=1}^J \dfrac{\partial e_{j'}(z_t, t)}{\partial t}}{(\sum_{j'=1}^J e_{j'}(z_t, t))^2}\right|\\
        \le&\dfrac{e_j(z_t, t)\dfrac{1+\lambda_{\max}}{\min\{\lambda^2_{\min},\frac{1}{4}\}}(\|z_t\|_2 + M)^2}{\sum_{j'=1}^J e_{j'}(z_t, t)}\\
        &+\dfrac{e_j(z_t, t)\sum_{j'=1}^J e_{j'}(z_t, t)\dfrac{1+\lambda_{\max}}{\min\{\lambda^2_{\min},\frac{1}{4}\}}(\|z_t\|_2 + M)^2}{(\sum_{j'=1}^J e_{j'}(z_t, t))^2}\\
        =&\dfrac{2+2\lambda_{\max}}{\min\{\lambda^2_{\min},\frac{1}{4}\}}(\|z_t\|_2 + M)^2w_j(z_t,t), 
    \end{aligned}
\end{equation}
which completes the proof.
\end{proof}

With Lemmas~\ref{lem_dwidzt} and~\ref{lem_wi_t}, we have established explicit upper bounds on the partial derivatives of the weight function $w_j(z_t,t)$ in Eq.~\eqref{eq:wi} with respect to both the spatial variable $z_t$ and the temporal variable $t$. Equipped with these tools, we are now ready to analyze the global error of the reverse ODE in Eq.~\eqref{eq:ode3.1.1}. 

\begin{thm}\label{thm_euler}
    Under Assumption \ref{assum1} and \ref{assum2}, Lemma \ref{lem_dwidzt} and \ref{lem_wi_t}, consider the probability flow ODE in Eq.~\eqref{eq:ode3.1.1} with initial condition $z_1$ at $t=1$. The global error between $z_0$ and numerical solution $z^K$ in Eq.~\eqref{eq:euler} is bounded by 
    \begin{equation}\label{thm_euler_eq1}
    \begin{aligned}
        \left\|z_0 - z^K \right\|_2 \le& C\exp\left(\frac{1+2\lambda_{\max}+8M^2}{4\min\{\lambda^2_{\min},\frac{1}{4}\}}\right)\\
        &\cdot\frac{(1+\lambda_{\max})^2(1+M)\Big(1 + \frac{\max\{\sqrt{\lambda_{\max}}, 1\}}{\min\{\sqrt{\lambda_{\min}}, \frac14\}}(M+\|z_1\|_2) + M\Big)^2}{(1+2\lambda_{\max}+8M^2)\min\{\lambda_{\min},\frac{1}{2}\}}h,
    \end{aligned}      
    \end{equation}
where $M$ is the bound in Assumption~\ref{assum2}, $\lambda_{\min}$ and $\lambda_{\max}$ are the smallest and largest eigenvalues of the covariance matrix $\Sigma$ from the Gaussian mixture model in Eq.~\eqref{eq:gmm}, $h$ is the time step in Eq.~\eqref{eq:euler}, and $C$  is a constant independent of $\lambda_{\max}, \lambda_{\min}$$, z_1, M, h$ and $d$.
\end{thm}
\begin{proof}
For notational simplicity, we define the following drift function
\begin{equation*}
\begin{aligned}
    f(z_t, t)& :=\frac{1}{2}(1-2(1-t)\Sigma)((1-t)^2\Sigma +t\mathbf{I}_d)^{-1}z_t\\
    & \quad\;\;- \frac{1+t}{2}((1-t)^2\Sigma +t\mathbf{I}_d)^{-1}\sum_{j=1}^{J}z_0^jw_j(z_t,t),
\end{aligned}
\end{equation*}
such that the reverse ODE in Eq.~\eqref{eq:ode3.1.1} can be rewritten as 
\begin{equation*}
    \frac{d z_t}{dt} = f(z_t, t).
\end{equation*}
We first estimate the Jacobian of \( f \) with respect to \( z_t \):
\begin{equation*}
\begin{aligned}
    \frac{\partial f(z_t, t)}{\partial z_t}
    =&\frac{1}{2}(\mathbf{I}_d-2(1-t)\Sigma)((1-t)^2\Sigma +t\mathbf{I}_d)^{-1}\\
    &-\frac{1+t}{2}((1-t)^2\Sigma +t\mathbf{I}_d)^{-1}\sum_{j=1}^{J}z_0^j\frac{\partial w_j(z_t,t)}{\partial z_t}.
\end{aligned}
\end{equation*}
Taking the $\ell_2$ norm, we obtain
\begin{equation*}
\begin{aligned}
    \left\|\frac{\partial f(z_t, t)}{\partial z_t}\right\|_2
    \le&\frac{1}{2}\|(\mathbf{I}_d-2(1-t)\Sigma)\|_2\|((1-t)^2\Sigma +t\mathbf{I}_d)^{-1}\|_2\\
    &+\frac{1+t}{2}\left\|((1-t)^2\Sigma +t\mathbf{I}_d)^{-1}\right\|_2\left\|\sum_{j=1}^{J}z_0^j\frac{\partial w_j(z_t,t)}{\partial z_t}\right\|_2\\
    \le&\frac{1+2\lambda_{\max}}{2\min\{\lambda_{\min},\frac{1}{2}\}} + \frac{M}{\min\{\lambda_{\min},\frac{1}{2}\}}\sum_{j=1}^{J}\frac{2M}{\min\{\lambda_{\min},\frac{1}{2}\}}w_j(z_t,t),\\
\end{aligned}
\end{equation*}
where the last inequality follows Eq.~\eqref{lem_dwidzt_eq1} and Lemma \ref{lem_dwidzt}. Since $\sum_{j=1}^{J}$ $w_j(z_t,t) = 1$, we have
\begin{equation}\label{dfdzt}
\begin{aligned}
    \left\|\frac{\partial f(z_t, t)}{\partial z_t}\right\|_2
    \le&\frac{1+2\lambda_{\max}}{2\min\{\lambda_{\min},\frac{1}{2}\}} + \frac{2M^2}{\min\{\lambda^2_{\min},\frac{1}{4}\}}\\
    \le&\frac{1+2\lambda_{\max}}{4\min\{\lambda^2_{\min},\frac{1}{4}\}} + \frac{2M^2}{\min\{\lambda^2_{\min},\frac{1}{4}\}}
    = \frac{1+2\lambda_{\max}+8M^2}{4\min\{\lambda^2_{\min},\frac{1}{4}\}},
\end{aligned}
\end{equation}
which shows $f(z_t, t)$ is Lipschitz continuous in $z_t$ with Lipschitz constant $\frac{1+2\lambda_{\max}+8M^2}{4\min\{\lambda^2_{\min},\frac{1}{4}\}}$. Next, we estimate the second-order derivative:
\begin{equation}\label{zt''}
\begin{aligned}
    \left\|\frac{d^2 z_t}{dt^2}\right\|_2 
    &= \left\|\frac{\partial f(z_t, t)}{\partial z_t} \cdot \frac{d z_t}{dt} + \frac{\partial f(z_t, t)}{\partial t} \right\|_2 \\
    &\le \left\|\frac{\partial f(z_t, t)}{\partial z_t}\right\|_2 \cdot \left\|f(z_t, t)\right\|_2 + \left\|\frac{\partial f(z_t, t)}{\partial t}\right\|_2.
\end{aligned}
\end{equation}
We estimate the norm of $f(z_t, t)$ as follows:
\begin{equation}\label{f_norm}
\begin{aligned}
    \left\|f(z_t, t)\right\|_2 =& \Bigg\| \frac{1}{2}(\mathbf{I}_d - 2(1-t)\Sigma)((1-t)^2\Sigma + t\mathbf{I}_d)^{-1}z_t \\
    &- \frac{1+t}{2}((1-t)^2\Sigma + t\mathbf{I}_d)^{-1}\sum_{j=1}^{J}z_0^j w_j(z_t,t) \Bigg\|_2 \\
    \le& \frac{1}{2} \|\mathbf{I}_d - 2(1-t)\Sigma\|_2 \|((1-t)^2\Sigma + t\mathbf{I}_d)^{-1}\|_2 \|z_t\|_2 \\
    &+ \frac{1+t}{2} \left\|((1-t)^2\Sigma + t\mathbf{I}_d)^{-1}\right\|_2 \left\|\sum_{j=1}^{J}z_0^j w_j(z_t,t)\right\|_2 \\
    \le& \frac{(1+2\lambda_{\max})\|z_t\|_2}{2\min\{\lambda_{\min},\frac{1}{2}\}} + \frac{M}{\min\{\lambda_{\min},\frac{1}{2}\}}
    = \frac{(1+2\lambda_{\max})\|z_t\|_2+2M}{2\min\{\lambda_{\min},\frac{1}{2}\}}.
\end{aligned}
\end{equation}
Additionally, we can derive the partial derivative of $f(z_t, t)$ with respect to $t$:
%
\begin{equation}\label{thm_euler_dfdt}
\begin{aligned}
&\frac{\partial f(z_t, t)}{\partial t}\\
=&\Sigma ((1-t)^2\Sigma + t\mathbf{I}_d)^{-1} z_t\\
&-\frac{1}{2}(\mathbf{I}_d - 2(1-t)\Sigma)((1-t)^2\Sigma + t\mathbf{I}_d)^{-1}(\mathbf{I}_d - 2(1-t)\Sigma)((1-t)^2\Sigma + t\mathbf{I}_d)^{-1}z_t\\
&-\frac{1}{2}((1-t)^2\Sigma + t\mathbf{I}_d)^{-1}\sum_{j=1}^{J}z_0^j w_j(z_t,t)\\
&+\frac{1+t}{2}((1-t)^2\Sigma + t\mathbf{I}_d)^{-1}(\mathbf{I}_d - 2(1-t)\Sigma)((1-t)^2\Sigma + t\mathbf{I}_d)^{-1}\sum_{j=1}^{J}z_0^j w_j(z_t,t)\\
&-\frac{1+t}{2}((1-t)^2\Sigma + t\mathbf{I}_d)^{-1}\sum_{j=1}^{J}z_0^j \frac{\partial w_j(z_t,t)}{\partial t}.
\end{aligned}
\end{equation}
To bound the time derivative \( \frac{\partial f(z_t, t)}{\partial t} \) in Eq.~\eqref{thm_euler_dfdt}, we apply Assumption \ref{assum2}, Lemma \ref{lem_wi_t} and 
Eq.~\eqref{lem_dwidzt_eq1}, yielding
\begin{equation}\label{dfdt}
\begin{aligned}
\left\|\frac{\partial f(z_t, t)}{\partial t}\right\|_2
\le&\frac{\lambda_{\max}\|z_t\|_2}{\min\{\lambda_{\min},\frac{1}{2}\}} + \frac{(1+2\lambda_{\max})^2\|z_t\|_2}{2\min\{\lambda^2_{\min},\frac{1}{4}\}} + \frac{M}{2\min\{\lambda_{\min},\frac{1}{2}\}}\\
&+\frac{(1+2\lambda_{\max})M}{\min\{\lambda^2_{\min},\frac{1}{4}\}} +\dfrac{(2+2\lambda_{\max})(\|z_t\|_2 + M)^2M}{\min\{\lambda^3_{\min},\frac{1}{8}\}}\\
\le&\frac{\lambda_{\max}\|z_t\|_2}{4\min\{\lambda^3_{\min},\frac{1}{8}\}} + \frac{(1+2\lambda_{\max})^2\|z_t\|_2}{4\min\{\lambda^3_{\min},\frac{1}{8}\}} + \frac{M}{8\min\{\lambda^3_{\min},\frac{1}{8}\}}\\
&+\frac{(1+2\lambda_{\max})M}{2\min\{\lambda^3_{\min},\frac{1}{8}\}} +\dfrac{(2+2\lambda_{\max})(\|z_t\|_2 + M)^2M}{\min\{\lambda^3_{\min},\frac{1}{8}\}}\\
\le&C\frac{(1+\lambda_{\max})^2(1+M)(\|z_t\|_2 + M)^2}{\min\{\lambda^3_{\min},\frac{1}{8}\}},
\end{aligned}
\end{equation}
where $C$ is a constant independent of $\lambda_{\max}, \lambda_{\min}, \|z_t\|_2$ and $M$.
Combining Eqs.~\eqref{dfdzt}, \eqref{zt''}, \eqref{f_norm}, and~\eqref{dfdt}, we obtain
\begin{equation}\label{zt''2}
\left\|\frac{d^2 z_t}{dt^2}\right\|_2 
\le C\frac{(1+\lambda_{\max})^2(1+M)(1 + \|z_t\|_2 + M)^2}{\min\{\lambda^3_{\min},\frac{1}{8}\}}.
\end{equation}
Therefore, by the global error of the Euler scheme, together with Eq.~\eqref{dfdzt} and Eq.~\eqref{zt''2}, we have
\begin{equation}
\begin{aligned}
\left\|z_0 - z^K \right\|_2&\le\frac{\exp\left(\frac{1+2\lambda_{\max}+8M^2}{4\min\{\lambda^2_{\min},\frac{1}{4}\}}\right) - 1}{2\frac{1+2\lambda_{\max}+8M^2}{4\min\{\lambda^2_{\min},\frac{1}{4}\}}} C\frac{(1+\lambda_{\max})^2(1+M)(1 + \|z_t\|_2 + M)^2}{\min\{\lambda^3_{\min},\frac{1}{8}\}}h\\
&\le C\exp\left(\frac{1+2\lambda_{\max}+8M^2}{4\min\{\lambda^2_{\min},\frac{1}{4}\}}\right)\frac{(1+\lambda_{\max})^2(1+M)(1 + \|z_t\|_2 + M)^2}{(1+2\lambda_{\max}+8M^2)\min\{\lambda_{\min},\frac{1}{2}\}}h.
\end{aligned}
\end{equation}
Applying the bound of $\|z_t\|_2$ in Lemma \eqref{lem_zt_bound}, we obtain
\begin{equation}
    \begin{aligned}
        \left\|z_0 - z^K \right\|_2 \le& C\exp\left(\frac{1+2\lambda_{\max}+8M^2}{4\min\{\lambda^2_{\min},\frac{1}{4}\}}\right)\\
        &\cdot\frac{(1+\lambda_{\max})^2(1+M)(1 + \frac{\max\{\sqrt{\lambda_{\max}}, 1\}}{\min\{\sqrt{\lambda_{\min}}, \frac14\}}(M+\|z_1\|_2) + M)^2}{(1+2\lambda_{\max}+8M^2)\min\{\lambda_{\min},\frac{1}{2}\}}h,
    \end{aligned}      
\end{equation}
where $C$  is a constant independent of $\lambda_{\max}, \lambda_{\min}, \|z_1\|_2, M, h$ and  dimension $d$. The proof is complete. 
\end{proof}


\subsubsection{The expected error bound for the reverse ODE}
\label{sec:ode_error}
We now analyze the expected global error of the reverse ODE with a Gaussian initial condition $z_1\sim\mathcal{N}(0,\mathbf{I}_d)$. Since the integration of Eq.~\eqref{eq:ode3.1.1} proceeds backward from $t=1$ to $t=0$, different realizations of $z_1$ produce different pathwise global errors, see Eq.~\eqref{thm_euler_eq1} in Theorem \ref{thm_euler}. Therefore, it is natural to quantify the error in expectation with respect to the distribution of $z_1$. We are now ready to state our main result.

\begin{thm}\label{thm_main1}
Under Assumptions~\ref{assum1}, \ref{assum2}, consider the probability flow ODE in Eq.~\eqref{eq:ode3.1.1} with initial condition $z_1\sim\mathcal{N}(0,\mathbf{I}_d)$. Let $z_0$ denote the exact solution at $t=0$, and let $z^{K}$ be the numerical solution produced by the Euler scheme in Eq.~\eqref{eq:euler}. Then the expected $\ell_2$-norm of the global error satisfies
    \begin{equation}\label{eq:l2_err}
        \mathbb{E}_{Z_1}\left\|z_0-z^K\right\|_2\le Cdh,
    \end{equation}
where $C$  is a constant independent of the time step $h$ and the dimension $d$.
\end{thm}

\begin{proof}
By Theorem~\ref{thm_euler}, for a fixed $z_1$, the resulting global approximation error for the discretized dynamics \eqref{eq:euler} applied to the probability flow ODE in Eq.~\eqref{eq:ode3.1.1} is bounded as follows:
\begin{equation}
    \begin{aligned}
        \left\|z_0 - z^K \right\|_2 \le& C\exp\left(\frac{1+2\lambda_{\max}+8M^2}{4\min\{\lambda^2_{\min},\frac{1}{4}\}}\right)\\
        &\cdot\frac{(1+\lambda_{\max})^2(1+M)(1 + \frac{\max\{\sqrt{\lambda_{\max}}, 1\}}{\min\{\sqrt{\lambda_{\min}}, \frac14\}}(M+\|z_1\|_2) + M)^2}{(1+2\lambda_{\max}+8M^2)\min\{\lambda_{\min},\frac{1}{2}\}}h,
    \end{aligned}      
\end{equation}
Taking the expectation with respect to $Z_1$, we obtain
\begin{equation}\label{eq:E(error_2)}
\begin{aligned}
\mathbb{E}_{Z_1}\left\|z_0-z^K\right\|_2
\le&Ch\mathbb{E}_{Z_1}\left[(\|z_1\|_2 + 1)^2 \right],
\end{aligned}
\end{equation}
where $C$  is a constant independent of the time step $h$ and the dimension $d$.
Since $Z_1 \sim \mathcal{N}(0,\mathbf{I}_d)$, $\|Z_1\|_2$ has the chi distribution with $d$ degrees of freedom, and
\begin{equation}
\mathbb{E}_{Z_1}[\|z_1\|_2^2]=d,
\qquad
\mathbb{E}_{Z_1}[\|z_1\|_2]
= \sqrt{2}\,\frac{\Gamma\big(\frac{d+1}{2}\big)}{\Gamma\big(\frac{d}{2}\big)}
\le \sqrt{d} .
\end{equation}
Substituting these into Eq.~\eqref{eq:E(error_2)} yields
\begin{equation}
\mathbb{E}_{Z_1}\|z_0 - z^K\|_2 \le Ch\big(d + 2\sqrt{d} + 1\big) \le Cdh,
\end{equation}
where $C$ is independent of the time step $h$ and the dimensionality $d$.
\end{proof}

Having established an expected global error bound in the $\ell_2$ norm, we next derive an analogous estimate in the $\ell_\infty$ norm. The $\ell_\infty$ control is complementary to the $\ell_2$ result and  provides a coordinatewise guarantee and is particularly useful when one seeks uniform accuracy across components. In the following, we quantify the expected $\ell_\infty$-norm of the global discretization error under the same setting.

\begin{thm}\label{thm_main2}
Assume Assumptions~\ref{assum1}, \ref{assum2} and  the Gaussian mixture model in Eq.~\eqref{eq:gmm} has a diagonal covariance matrix $\Sigma$. 
 Then the expected $\ell_\infty$-norm of the global error satisfies
    \begin{equation}\label{eq:l_inf_err}
        \mathbb{E}_{Z_1}\left\|z_0-z^K\right\|_{\infty}\le C(\ln(d) + 1)h,
    \end{equation}
where $C$  is a constant independent of the time step $h$ and the dimension $d$.
\end{thm}
\begin{proof}
When the covariance matrix $\Sigma$ in \eqref{eq:gmm} is diagonal, all matrix-valued quantities that arise in the probability flow ODE and in the associated Euler analysis remain diagonal as well. For a diagonal matrix, the induced operator $\ell_2$-norm and $\ell_\infty$-norm coincide. Moreover, by Assumption~\ref{assum2} we have, for each $j=1,\dots,J$,
\[
\|z_0^j\|_\infty \le \|z_0^j\|_2 \le M,
\]
so the data bound can be stated equivalently in the $\ell_\infty$ norm. Combining these observations, every matrix norm estimate established earlier in the $\ell_2$ setting continues to hold verbatim when $\|\cdot\|_2$ is replaced by $\|\cdot\|_\infty$. In particular, we may invoke Theorem~\ref{thm_euler} to obtain a pathwise $\ell_\infty$ norm bound for the Euler discretization error for a fixed initial condition $z_1$
\begin{equation}
    \begin{aligned}
        \left\|z_0 - z^K \right\|_\infty \le& C\exp\left(\frac{1+2\lambda_{\max}+8M^2}{4\min\{\lambda^2_{\min},\frac{1}{4}\}}\right)\\
        &\cdot\frac{(1+\lambda_{\max})^2(1+M)(1 + \frac{\max\{\sqrt{\lambda_{\max}}, 1\}}{\min\{\sqrt{\lambda_{\min}}, \frac14\}}(M+\|z_1\|_\infty) + M)^2}{(1+2\lambda_{\max}+8M^2)\min\{\lambda_{\min},\frac{1}{2}\}}h,
    \end{aligned}      
\end{equation}
Taking the expectation with respect to $Z_1$, we obtain
\begin{equation}\label{Ez0-uk}
\begin{aligned}
\mathbb{E}_{Z_1}\left\|z_0-z^K\right\|_{\infty}
\le&Ch\mathbb{E}_{Z_1}\left[(\|z_1\|_{\infty} + 1)^2 \right]
\end{aligned}
\end{equation}
where $C$  is a constant independent of the time step $h$ and the dimension $d$.

The next step is to estimate the expectations $\mathbb{E}_{Z_1}\|z_1\|_{\infty}$ and $\mathbb{E}_{Z_1}\|z_1\|_{\infty}^2$. Let $z_1^i$ denote the $i$-th component of $z_1$, where $z_1^i,\ldots,z_1^d$ are independent and identically distributed as $\mathcal{N}(0, 1)$. Applying Jensen's inequality,  we have, for $t>0$, 
\begin{equation*}
\begin{aligned}
\exp(t\mathbb{E}_{Z_1}\|z_1\|_{\infty})\leq&\mathbb{E}_{Z_1}[\exp(t\|z_1\|_{\infty})] = \mathbb{E}_{Z_1}[\exp(t\max_{1\leq i \leq d}|z_1^i|)] =  \mathbb{E}_{Z_1}[\max_{1\leq i \leq d}\exp(t|z_1^i|)]\\
\leq&\mathbb{E}_{Z_1}\left[\sum_{i=1}^d\left(\exp(tz_1^i) + \exp(-tz_1^i)\right)\right] \\
=& \sum_{i=1}^d\left(\mathbb{E}_{Z_1}[\exp(tz_1^i)] + \mathbb{E}_{Z_1}[\exp(-tz_1^i)]\right) = 2d\exp(t^2/2).
\end{aligned}
\end{equation*}
Taking the natural logarithm on both sides, we obtain
\begin{equation*}
\begin{aligned}
    \mathbb{E}_{Z_1}\|z_1\|_{\infty}&\leq\frac{\ln(2d)}{t} + \frac{t}{2}, \quad \forall \ t>0.
\end{aligned}
\end{equation*}
Choosing the optimal value $t = \sqrt{2\ln(2d)}$ then yields
\begin{equation}\label{Ez1}
\begin{aligned}
\mathbb{E}_{Z_1}\|z_1\|_{\infty}&\leq \sqrt{2\ln(2d)}.
\end{aligned}
\end{equation}
Similarly, we estimate the second moment by noting
\begin{equation*}
\begin{aligned}
&\exp\left(\frac{1}{4}\mathbb{E}_{Z_1}\|z_1\|_{\infty}^2\right)\leq\mathbb{E}_{Z_1}\left[\exp\left(\frac{1}{4}\|z_1\|_{\infty}^2\right)\right] = \mathbb{E}_{Z_1}\left[\exp\left(\frac{1}{4}\max_{1\leq i \leq d}|z_1^i|^2\right)\right]\\
=&\mathbb{E}_{Z_1}\left[\max_{1\leq i \leq d}\exp\left(\frac{1}{4}|z_1^i|^2\right)\right]
\leq \mathbb{E}_{Z_1}\left[\sum_{i=1}^d\exp\left(\frac{1}{4}(z_1^i)^2\right)\right] = \sum_{i=1}^d\mathbb{E}_{Z_1}\left[\exp\left(\frac{1}{4}(z_1^i)^2\right)\right]\\
=& \sqrt{2}d.
\end{aligned}
\end{equation*}
Taking the natural logarithm of both sides in the above, we find
\begin{equation}\label{Ez1^2}
\begin{aligned}
    \mathbb{E}_{Z_1}\|z_1\|_{\infty}^2&\leq 4\ln(\sqrt{2}d).
\end{aligned}
\end{equation}
Combining Eqs.~\eqref{Ez0-uk}, \eqref{Ez1}, and \eqref{Ez1^2}, we conclude
\begin{equation*}
\mathbb{E}_{Z_1}\left\|z_0-z^K\right\|_{\infty}\le C(\ln(d) + 1)h,
\end{equation*}
as desired. This completes the proof.
\end{proof}

\subsection{Discussion on the error bounds} 
\subsubsection{The choice of the covariance matrix $\Sigma$}\label{sec:sigma}
In practice, choosing $\Sigma$ of the Gaussia mixture model in Eq.~\eqref{eq:gmm} with either extremely small or extremely large eigenvalues tends 
to degrade numerical accuracy. For example, when the covariance matrix is in a diagonal form $\Sigma = \sigma \mathbf{I}_d$, the analytical expression of the Lipschitz constant $L(\sigma)$ (as a function of $\sigma$) of the linear part of the reverse ODE in Eq.~\eqref{eq:ode3.1.1} is plotted in Fig.~\ref{fig:lipschitz_f}. We observe that  
%
the Lipschitz constant initially 
decreases as \(\sigma\) grows, reaches a favorable plateau, and then increases again. 
This captures the intrinsic trade-off in tuning the covariance scale, i.e., too small  
spread leads to stiffness, while too large spread weakens the contraction structure 
of the flow. Our numerical experiments in Section~4.1 verify this qualitative 
behavior and indicate that an intermediate range of \(\sigma\) achieves the most 
stable and accurate integration.



\subsubsection{The dimensionality dependence}
A typical origin of the curse of dimensionality in global error estimates for numerical ODE solvers is the exponential amplification from Grönwall-type bounds when the drift has a large Lipschitz 
\begin{wrapfigure}{r}{0.45\linewidth}
\vspace{-0.3cm}
  \centering
  \includegraphics[width=\linewidth]{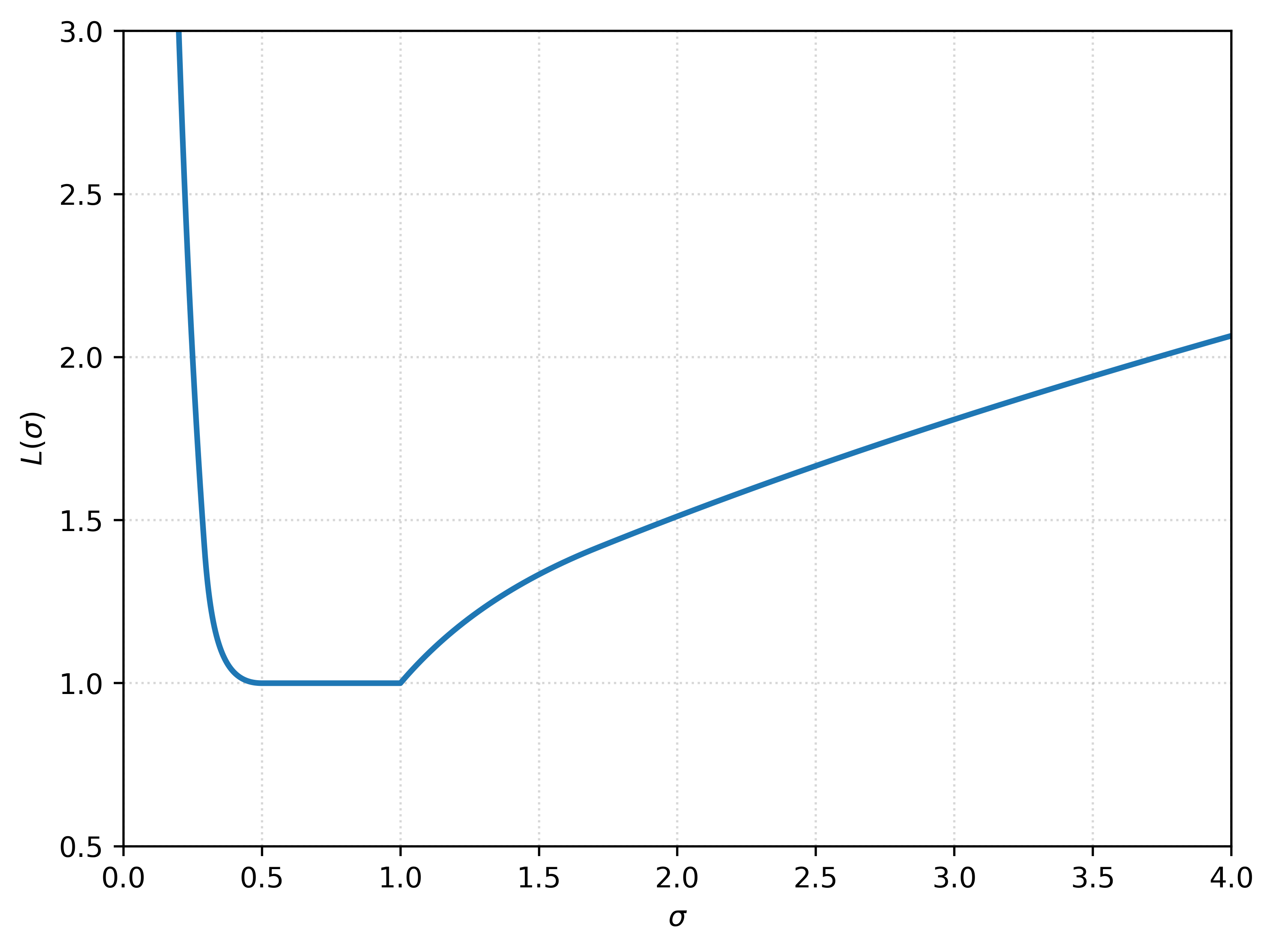}
  \vspace{-0.5cm}
  \caption{Upper bound of the Lipschitz constant $L(\sigma)$ of the linear part of the reverse ODE in Eq.~\eqref{eq:ode3.1.1}. The curve first decreases, and then increases. This trend will be confirmed empirically in Section~4.1 (Fig.~\ref{fig:sigma}).}
  \vspace{-0.4cm}
  \label{fig:lipschitz_f}
\end{wrapfigure}
constant.
If the effective Lipschitz constant $L$ grows with the ambient dimension $d$, the error prefactor can behave like $\exp(CL)$, where $C$ is a positive constant independent of time step h.

In our setting, the drift of Eq.~\eqref{eq:ode3.1.1} depends on the Gaussian–mixture weights $w_i(z_t,t)$. Lemma~\ref{lem_dwidzt} bounds the spatial sensitivity, which yields a uniform control of the Jacobian of the drift with respect to $z_t$ that depends on $M$ and on spectral quantities of $\Sigma$ but does not blow up exponentially with $d$. Specifically, the estimate Eq.~\eqref{dfdzt} inside Theorem~\ref{thm_euler} provides a bounded Lipschitz constant for the drift, thereby preventing the exponential in $d$ amplification in the Grönwall factor and ensuring a first-order global error bound with a dimension-stable prefactor.
The only remaining $d$-dependence appears when taking expectation over the Gaussian initial condition $z_1\sim\mathcal{N}(0,\mathbf{I}_d)$: one has $\mathbb{E}\|z_1\|_2^2=d$ and $\mathbb{E}\|z_1\|_2\le \sqrt{d}$, which leads to the linear dependence in $d$ of Theorem~\ref{thm_main1}. Thus the expected error scales like $\mathcal{O}(dh)$ rather than $O(\exp(d)h)$. We note that the constants deteriorate as $\lambda_{\min}\downarrow 0$, reflecting ill-conditioning of the mixture components; under Assumptions~\ref{assum2}, spectral control keeps the drift’s Lipschitz modulus bounded and averts dimension-induced exponential growth. The same mechanism extends to non-diagonal $\Sigma$ under comparable spectral bounds.


\section{Numerical experiments}
\label{sec:experiment}
We conduct a set of numerical experiments to validate the error bounds derived in Section~\ref{sec:analysis}. Since the analysis in Section~\ref{sec:analysis} accounts only for the reverse ODE discretization error under Assumption~\ref{assum1}, we choose the target distribution $p_X(x) = q_{Z_0}(z_0)$ to be a Gaussian mixture model defined as
\begin{equation}\label{eq:gmm_ex}
    q_{Z_0}(z_0) := \frac{1}{J} \sum_{j=1}^J 
    \phi\!\left(z_0; z_0^j, \Sigma\right),
\end{equation}
where $\phi(\cdot)$ denotes the Gaussian density function, $J$ is the number of mixture components, $\{z_0^j\}_{j=1}^J$ are the component means, and 
$
\Sigma
$
is a shared diagonal covariance matrix. The specific choices of $\{z_0^j\}_{j=1}^J$ and $\Sigma$ are specified in each experiment.
By adopting the Gaussian mixture model in~\eqref{eq:gmm_ex} as the target distribution, we eliminate the contribution of the data error term in~\eqref{eq:error_split}. This allows us to isolate the reverse ODE discretization error and directly assess the sharpness of the theoretical error bounds established in Section~\ref{sec:analysis}.

%

\subsection{Verification of the discretization error bounds of the reverse ODE}\label{sec:experiment1}
we verify the error bounds of the reverse ODE, focusing on how the error depends on the dimension, the time step size, and the covariance structure of the Gaussian mixture model. To quantify the discretization error, we use second-order Heun method with a small step size $h_{\mathrm{ref}}=10^{-3}$  as a reference solution. In all experiments we consider a Gaussian mixture with $J=10$ modes. The mean vectors $\{z_0^{j}\}_{j=1}^{J}$ in Eq.~\eqref{eq:gmm_ex} are drawn differently depending on the evaluation norm.
For the experiments verifying error bounds in $\ell_2$ norm, we sample $\{z_0^{j}\}_{j=1}^{J}$ independently and uniformly from the ball of radius $M$ in $\mathbb{R}^d$, while for the experiments verifying the error bounds in $\ell_\infty$ norm, we sample independently and uniformly from the hypercube $[-M,M]^d$. The expectation of the error bounds in Theorems~\ref{thm_main1} and \ref{thm_main2} is approximated by averaging the errors of 1000 trajectories. 

\vspace{0.2cm}
\paragraph{Verifying the dimension dependence}
Fig.~\ref{fig:d} demonstrates the dimension dependence estimated in Theorems~\ref{thm_main1} and~\ref{thm_main2}. 
The constant in Assumption~\ref{assum2} is set to $M = 1$, and the covariance matrix in Eq.~\eqref{eq:gmm_ex} is chosen as a full symmetric positive definite (SPD) matrix
$\Sigma = U \Lambda U^\top$, where $U\in\mathbb{R}^{d\times d}$ is a random orthogonal matrix and
$\Lambda=\mathrm{diag}(\sigma_1,\ldots,\sigma_d)$ collects the eigenvalues of $\Sigma$.
We generate $\{\sigma_k\}_{k=1}^d$ by periodically repeating the pattern $(0.1,0.2,0.3,0.4,0.5)$ until reaching dimension $d$.
We observe that the $\ell_2$ error increases approximately linearly with the dimension $d$, and the $\ell_\infty$ error exhibits an $\mathcal{O}(\log d)$ growth rate. These scaling behaviors are consistent with the dimension dependence predicted by Theorems~\ref{thm_main1} and~\ref{thm_main2}.
\begin{figure}[h!]
\vspace{-0.3cm}
\centering
\includegraphics[width=0.4\linewidth]{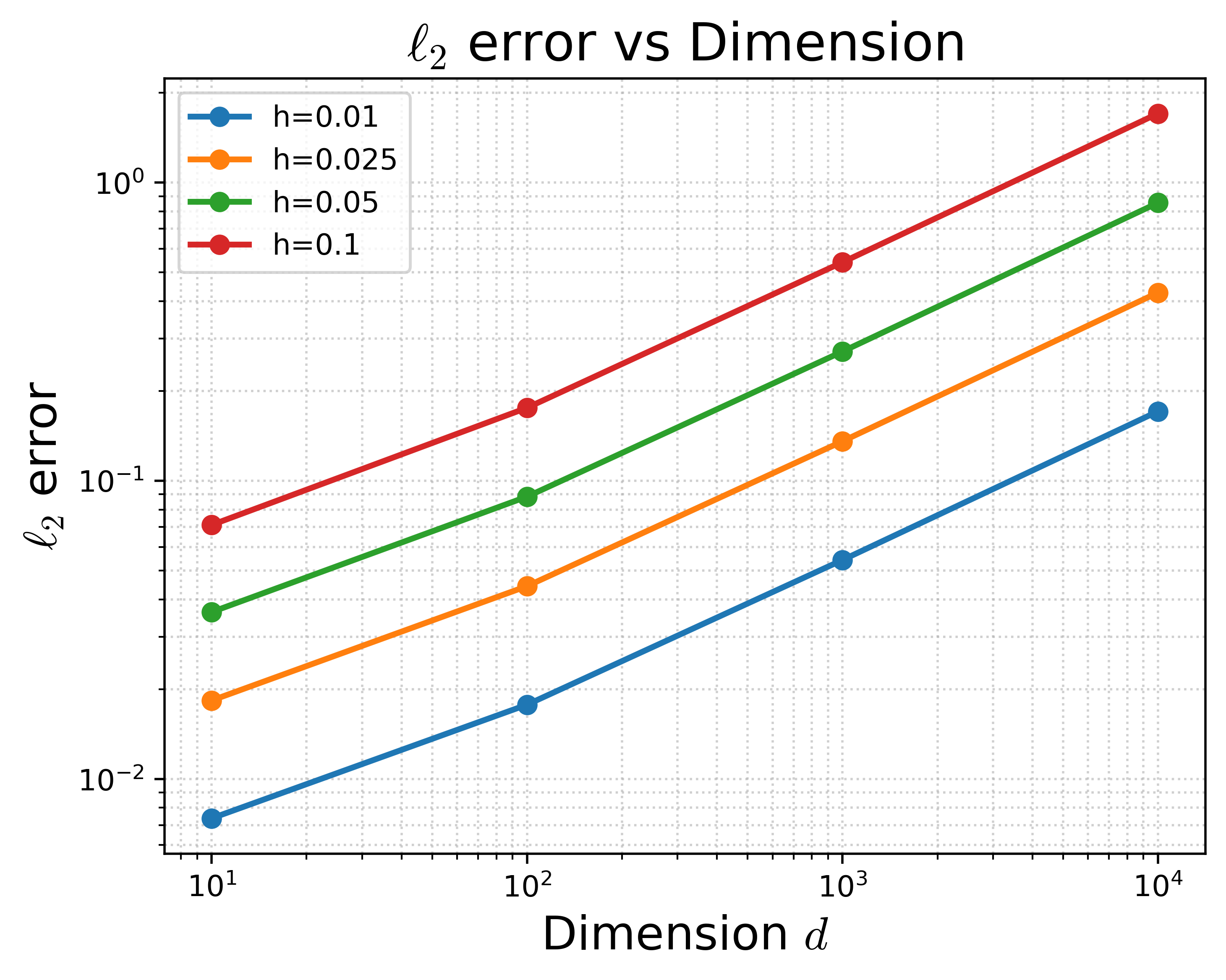}
\includegraphics[width=0.4\linewidth]{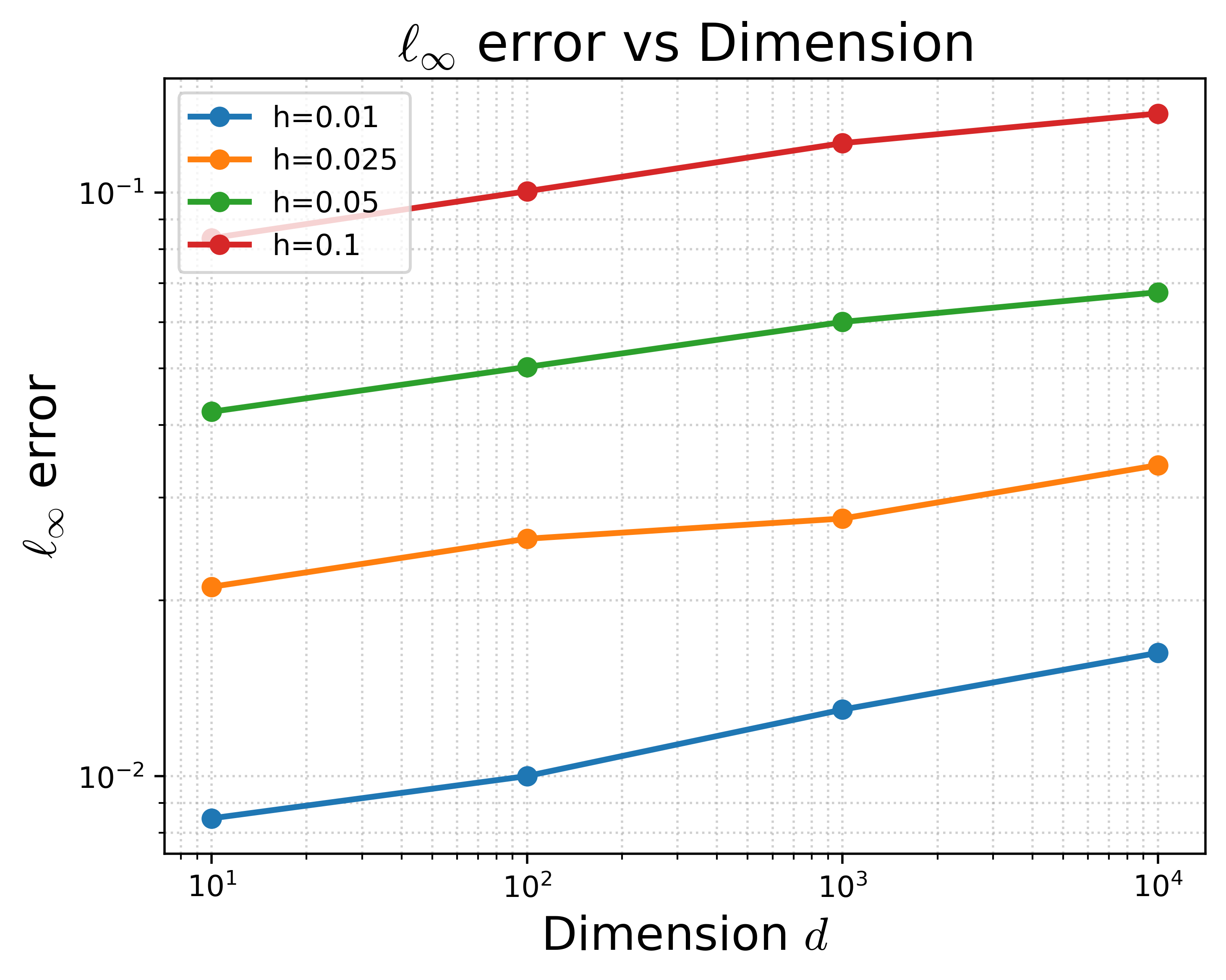}
\vspace{-0.2cm}
\caption{Demonstrating the error versus the dimension $d$ for $h\in \{0.1,0.05,0.025,0.01\}$ and $d \in \{10,100,1000,10000\}$. The $\ell_2$ error grows approximately on the order of $\mathcal{O}(d)$ and the $\ell_\infty$ error grows approximately on the order of $\mathcal{O}(\log d)$, verifying the dimension dependence in Eq.~\eqref{eq:l2_err} and Eq.~\eqref{eq:l_inf_err}.}
\label{fig:d}
\vspace{-0.5cm}
\end{figure}

\paragraph{Verifying the first-order temporal convergence rate}
We now verify the convergence rate with respect to the time step size $h$ presented in Theorem \ref{thm_main1} and Theorem \ref{thm_main2}. 
The time step size is set to $h \in \{0.01,0.02,0.05,0.1,0.2\}$, and the dimension is set to $d\in\{10,100,1000,10000\}$. The bound defined in Assumption \ref{assum2} is set to $M=1$, and We use the same full SPD covariance setting as above: $\Sigma = U\Lambda U^\top$ with $\Lambda=\mathrm{diag}(\sigma_1,\ldots,\sigma_d)$, where $\{\sigma_k\}_{k=1}^d$ repeats the pattern $(0.1,0.2,0.3,0.4,0.5)$.
Fig.~\ref{fig:h} shows that our method achieves a perfect first-order convergence rate with respect to the time step size $h$ in both $\ell_2$ and $\ell_\infty$ norm, which demonstrates the sharpness of our error bounds. 
\begin{figure}[h!]
\vspace{-0.2cm}
\centering
\includegraphics[width=0.4\linewidth]{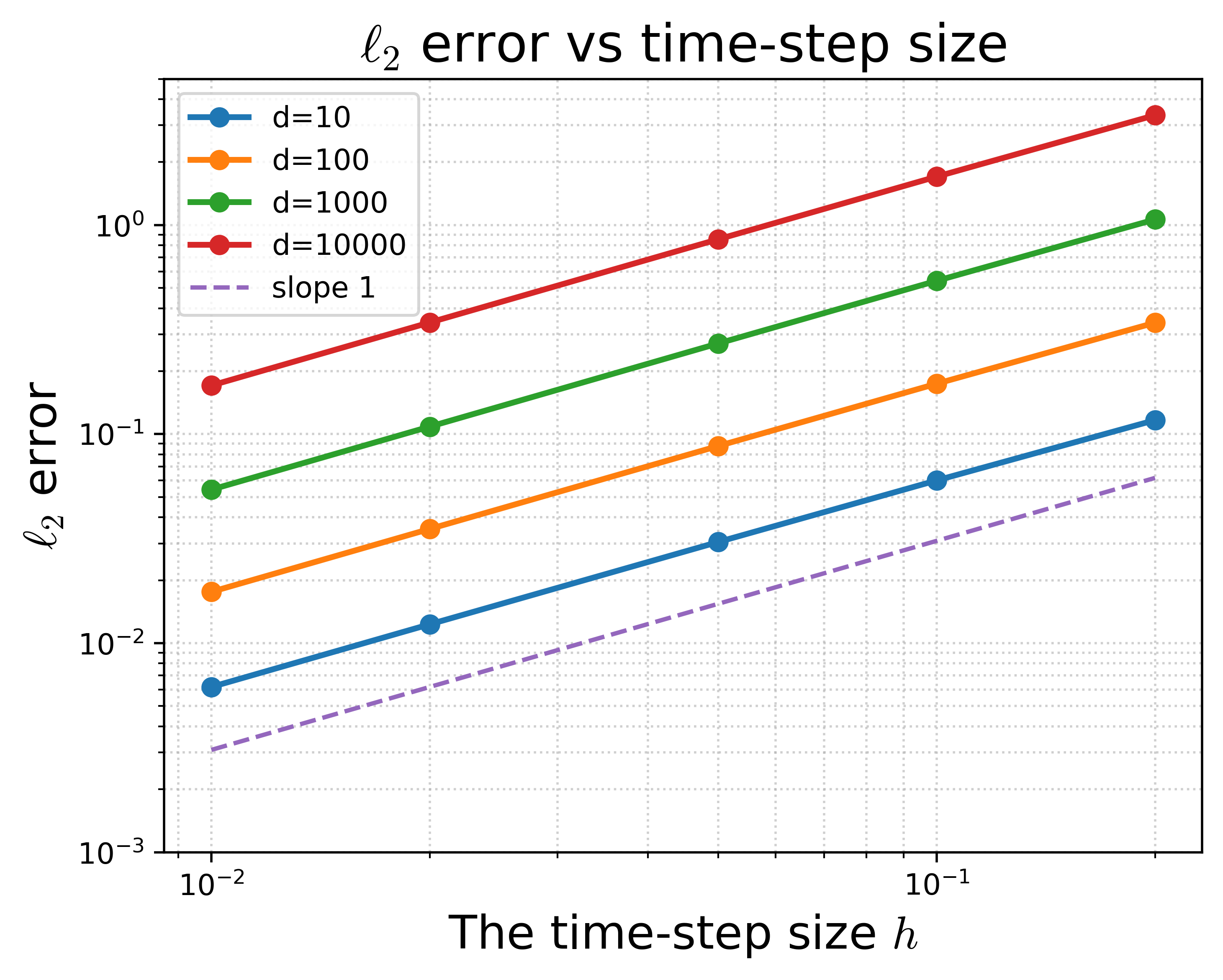}
\includegraphics[width=0.4\linewidth]{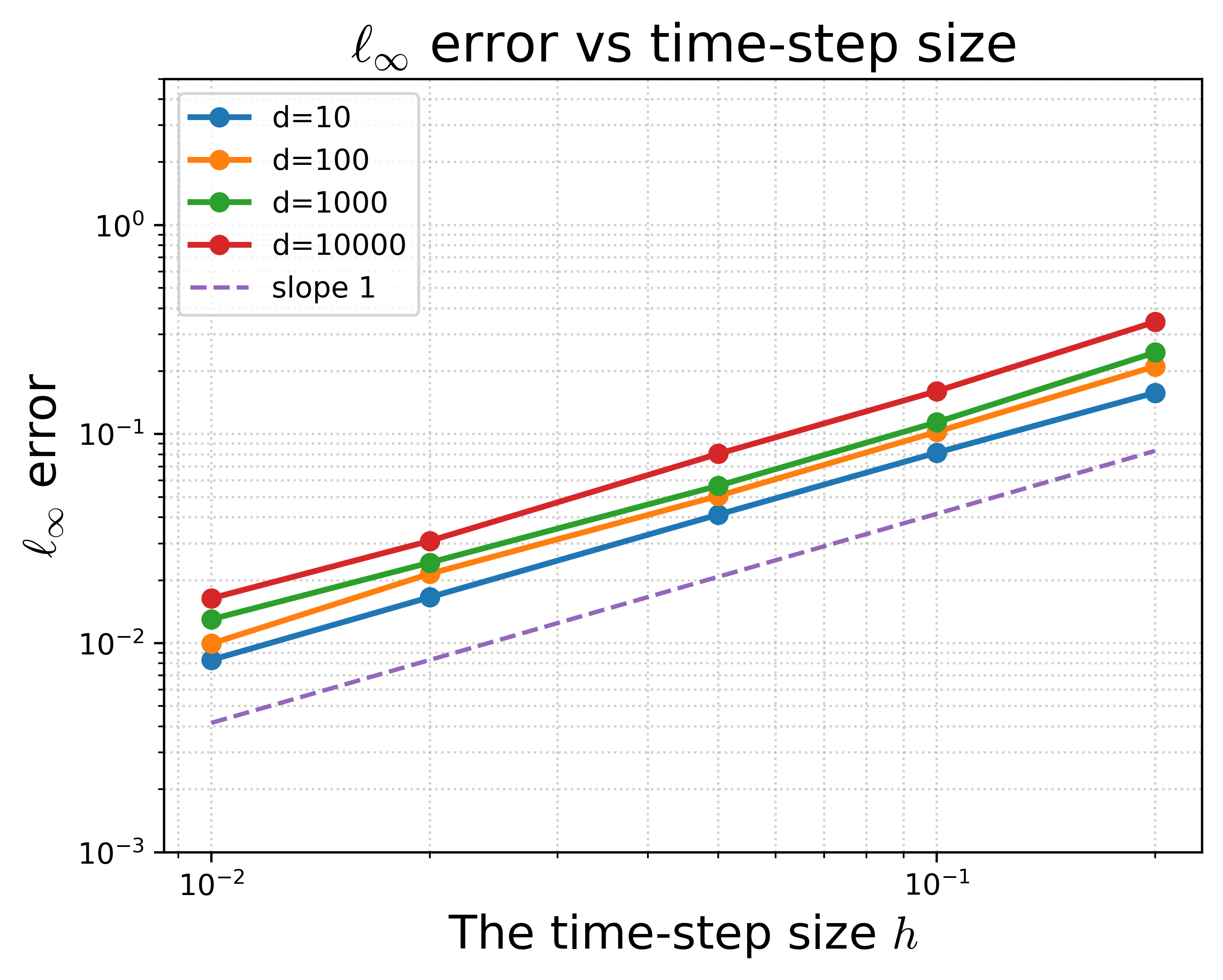}
\vspace{-0.2cm}
\caption{
Demonstrating the error versus the time-step size $h$ in ODE discretization for $h\in \{0.01, 0.02, 0.05, 0.1, 0.2\}$ and $d \in \{10,100,1000,10000\}$. Both the $\ell_2$ error and the $\ell_\infty$ error achieves perfect first-order convergence rate with respect to $h$ as shown in our error bounds in Eq.~\eqref{eq:l2_err} and Eq.~\eqref{eq:l_inf_err}.}
\label{fig:h}
\vspace{-0.5cm}
\end{figure}

\paragraph{Verifying the dependence on the smoothing constant $\sigma$}
We now verify the dependence of the error bound derived in Theorem \ref{thm_euler} on the smoothing constant $\sigma$. The covariance matrix in Eq.~\eqref{eq:gmm_ex} is chosen to be isotropic, $\Sigma=\sigma I_d$. We vary $\sigma\in\{0.1,0.2,0.3,0.4,0.5,0.6\}$ while fixing $d=10$, $h=0.01$, and $M_0=1$.
Fig.~\ref{fig:sigma} 
shows that the $\ell_2$ and $\ell_\infty$ errors decrease for small $\sigma$ and then increase as $\sigma$ grows, which is consistent with the discussion in Section \ref{sec:sigma} and the plot in Fig.~\ref{fig:lipschitz_f}.
\begin{figure}[h!]
\vspace{-0.2cm}
  \centering
  \includegraphics[width=0.4\linewidth]{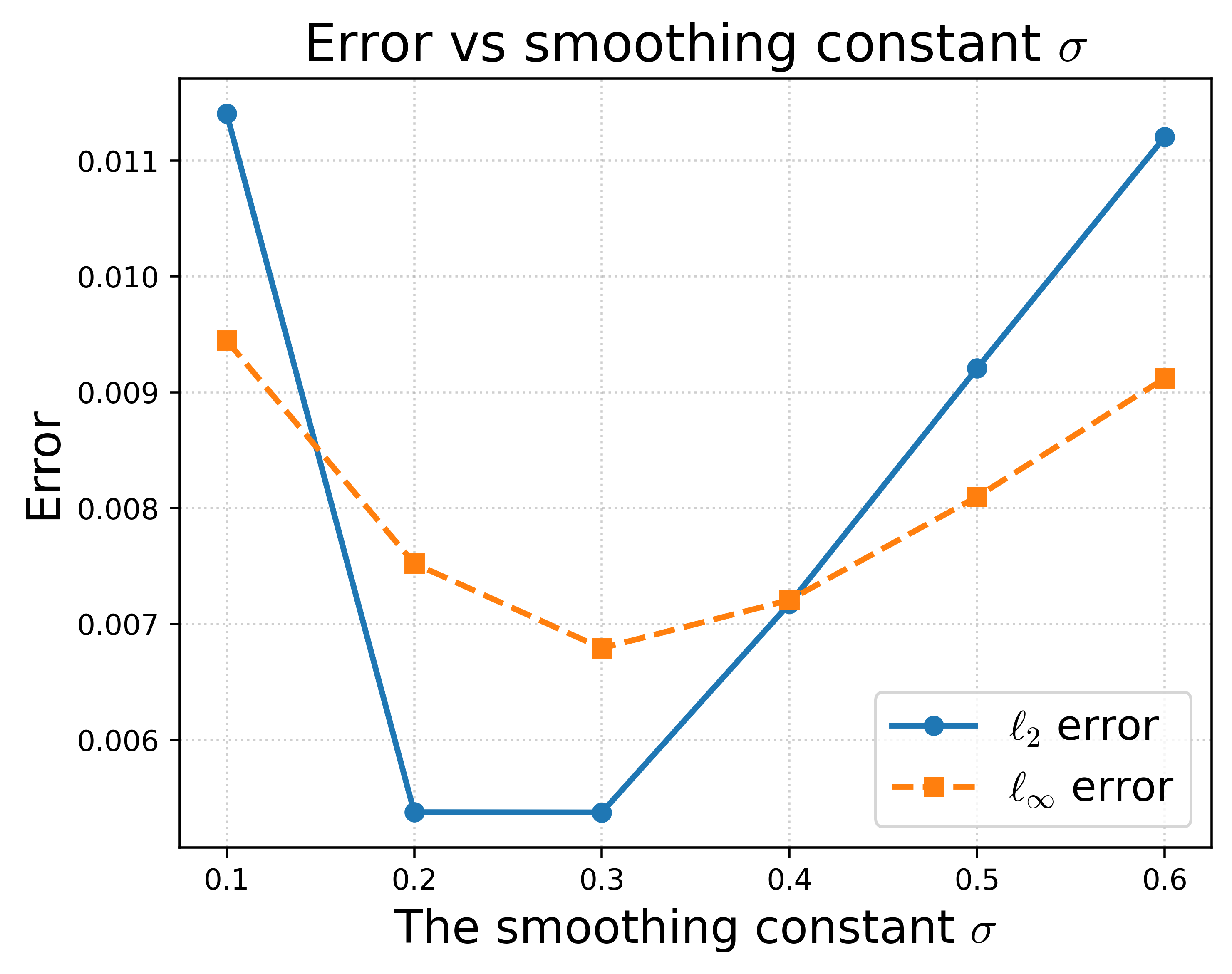}
  \vspace{-0.2cm}
  \caption{Error vs.\ smoothing constant $\sigma$ in Eq.~\eqref{eq:gmm_ex} for fixed $d=10$, $h=0.01$, $M=1$. The $\ell_2$ and $\ell_\infty$ errors decrease for small $\sigma$ and then increase as $\sigma$ grows, which is consistent with the discussion in Section \ref{sec:sigma} and the plot in Fig.~\ref{fig:lipschitz_f}.}
  \label{fig:sigma}
  \vspace{-0.7cm}
\end{figure}


\subsection{Verification of the supervised learning error}\label{sec:experiment2}
To demonstrate the practical effectiveness of the supervised learning strategy used to construct the generative model in~\eqref{eq:transport}, we present additional numerical experiments to assess the relative contributions of the reverse ODE discretization error and the supervised learning error in Eq.~\eqref{eq:error_split}, and to determine which error source is dominant in practice. 

\paragraph{Experiment setup} We use the reverse ODE with the exact score function to generate training pairs, and then train neural networks to directly learn the noise-to-sample map $z_1 \mapsto z_0$. We still use Eq.~\eqref{eq:gmm_ex} as the target distribution while setting $J=10$. 
Each Gaussian component's mean vector has coordinates drawn independently from $\{-1, 1\}$ with equal probability, and we use a shared diagonal covariance matrix with diagonal entries drawn uniformly from $[0.2, 0.4]$. We test across dimensions $d \in \{8, 16, 32, 64, 128, 256, 512, 1024, 2048, 4096, 8192\}$. To generate labeled data in the training set $\mathcal{D}_{\rm train}$, we sample $z_1 \sim \mathcal{N}(0, I_d)$ from the standard Gaussian and solve for the corresponding $z_0$ through the reverse ODE. The reference solution is generated with $h=0.001$ and the approximate solution is generated with $h= 0.01$. 
%
For each dimension, we generate 100,000 samples and split them into 80,000 training samples, 10,000 validation samples, and 10,000 test samples. We repeat each experiment 5 times with different random splits to assess variability. 
%
We compare two Multi-layer Perceptron (MLP) architectures with different capacities: MLP-1024 and MLP-2048, which have hidden dimensions of 1024 and 2048, respectively. Both architectures consist of two hidden layers with GELU activation functions. The models are trained using the AdamW optimizer with learning rate $10^{-3}$ and weight decay $10^{-4}$. Training continues until the validation loss fails to improve for 50 consecutive epochs (early stopping), at which point the model with the lowest validation loss is selected. 

\paragraph{Results and discussion}
Fig.~\ref{fig:supervised_error} shows the root mean squared error (RMSE) for different model configurations across dimensions. For each dimension, we compare two types of errors on both training and test sets, i.e., the reverse ODE discretization error and the neural network approximation error. We observe that the reverse ODE discretization error (shown in red) remains consistently small, ranging from approximately 0.01 to 0.05 in RMSE. This error grows slowly with dimension and remains nearly flat across the tested scenarios. 
%
%
Across all test cases, the total sampling error with respect to the true distribution is dominated by neural network approximation error, which exceeds the ODE discretization error by one to two orders of magnitude, as clearly visible in the log-scale plot (right panel).
This demonstrates that improving ODE solvers beyond a reasonable accuracy threshold yields diminishing returns. Instead, research efforts should prioritize three directions. First, scaling network architectures to handle high-dimensional noise to sample mappings. Second, improving data efficiency through better training procedures or data augmentation. Third, developing architectures that maintain strong generalization as the dimension increases.
\begin{figure}[h!]
\vspace{-0.5cm}
  \centering
\includegraphics[width=.9\linewidth]{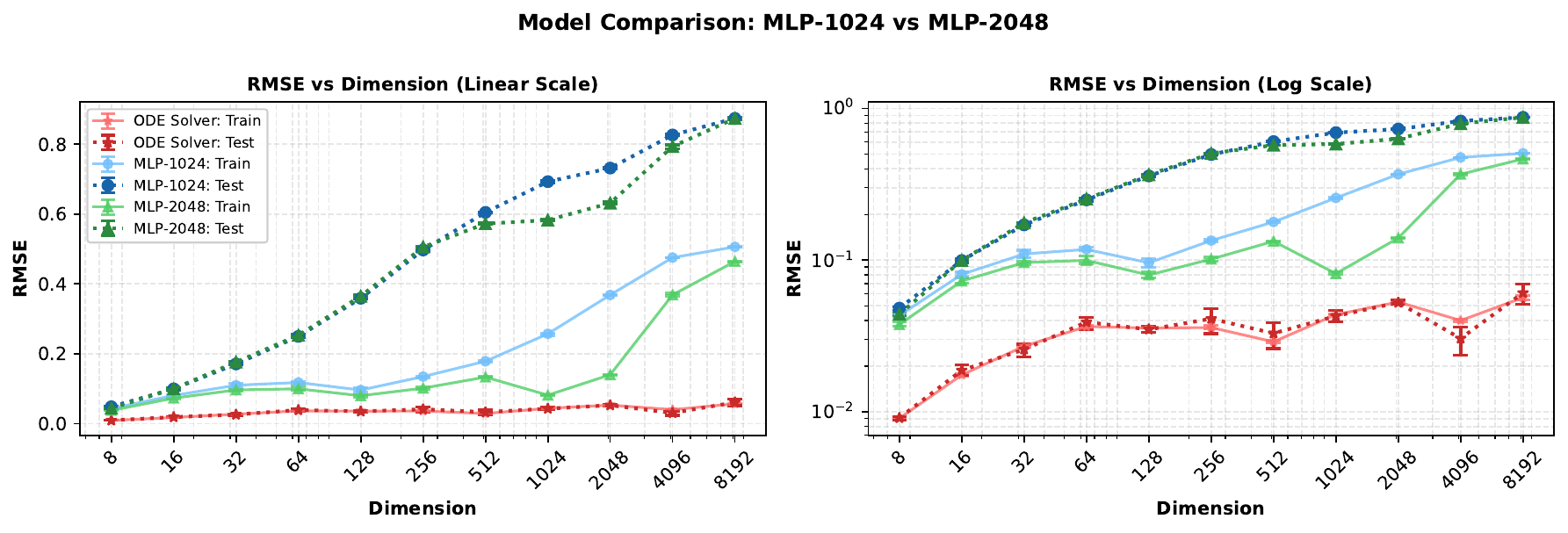}
\vspace{-0.2cm}
  \caption{Root mean squared error (RMSE) comparison across dimensions for supervised learning of the noise-to-sample mapping. The red curves show the reverse ODE discretization error, blue and green curves show neural network prediction errors (MLP-1024 in blue and MLP-2048 in green) on training (solid) and test (dotted) sets. Error bars show the 10th to 90th percentile range across 5 runs. Across all test cases, the total sampling error with respect to the true distribution is dominated by neural network approximation error, which exceeds the ODE discretization error by one to two orders of magnitude.}
  \label{fig:supervised_error}
  \vspace{-0.9cm}
\end{figure}

\section{Conclusion}\label{sec:con}
In this work, we developed a rigorous and numerically verifiable error analysis for a class of training-free diffusion models used to generate labeled data for supervised learning of end-to-end generative samplers. By decomposing the total error into data, discretization, and supervised learning components, we isolated the key source of theoretical uncertainty and focused on the discretization error arising from the reverse-time diffusion process when a Gaussian mixture model is treated as the target distribution.
A central feature of our analysis is the exploitation of analytically available score functions for Gaussian mixture models, which allows us to completely avoid the propagation of score-function approximation errors through the diffusion process. As a result, we recover classical convergence rates associated with standard ODE discretization schemes and obtain error bounds with favorable dependence on the problem dimensionality. Importantly, the derived estimates are not only theoretically rigorous but also directly verifiable through numerical experiments, both with respect to time-step size and dimensionality, thereby addressing a key gap between existing diffusion model theory and observed numerical behavior.

The framework presented in this paper provides a principled foundation for understanding the accuracy and scalability of training-free diffusion models in supervised generative learning. Beyond the specific setting considered here, the proposed error-splitting strategy offers a flexible blueprint for analyzing more general training-free and hybrid diffusion-based methods. Future work will extend the analysis to broader classes of target distributions, higher-order discretization schemes, and adaptive time-stepping strategies, as well as investigate the interaction between diffusion-based data generation and generalization error in supervised learning.

\bibliographystyle{siamplain}
\bibliography{references}
\end{document}